\documentclass[12pt]{amsart}

\usepackage{mathrsfs}
\usepackage{amsrefs}
\usepackage{amsthm}
\usepackage{enumerate}
\usepackage{graphicx}
\usepackage{graphicx} 
\usepackage{scrextend}
\usepackage{amsfonts, amsmath, amssymb,amsthm,comment}  
\usepackage{times,enumerate}
\usepackage{mathdots}%
\usepackage{nccmath}
\usepackage{mathtools}
\usepackage{ulem}
\usepackage{enumitem}
\usepackage{multicol}
		{\end{pmatrix}\end{medsize}}%
\usepackage{dsfont,bbm}
\usepackage{rotating}
\usepackage{lscape}
\usepackage{url}
\usepackage{multicol}
\usepackage{thmtools, thm-restate}
\usepackage{blkarray}
\usepackage{multicol}
\usepackage[a4paper,margin=2.5cm]{geometry}
\usepackage{hyperref}
\usepackage{cancel}
\usepackage{multirow}
\usepackage[baseline]{euflag}

\usepackage[section]{algorithm}
\usepackage{algpseudocode}

\usepackage{pifont}
\usepackage{tikz}
\usetikzlibrary{topaths}
\tikzstyle{every picture}+=[remember picture,inner xsep=0,inner ysep=0.25ex]
\usetikzlibrary{calc}

\DeclareFontFamily{U}{mathx}{\hyphenchar\font45 }
\DeclareFontShape{U}{mathx}{m}{n}{
	<5><6><7><8><9><10><10.95><12><14.4><17.28><20.74><24.88> mathx10
}{}
\DeclareSymbolFont{mathx}{U}{mathx}{m}{n}
\DeclareMathAccent{\widecheck}{0}{mathx}{"71}

\usepackage{kbordermatrix}
\renewcommand{\kbldelim}{(}
\renewcommand{\kbrdelim}{)}
\renewcommand{\kbrowstyle}{\displaystyle}
\renewcommand{\kbcolstyle}{\displaystyle}

\makeatletter
\newcommand*{\sublabel}[1]{%
	\let\old@currentlabel\@currentlabel%
	\renewcommand{\@currentlabel}{\theenumii}%
	\label{#1}%
	\let\@currentlabel\old@currentlabel%
}
\makeatother
\allowdisplaybreaks

\graphicspath{ {./figs/} }

\hypersetup{
	colorlinks,
	linkcolor={blue},
	citecolor={blue},
	urlcolor={blue}
}

\DeclareMathOperator{\Span}{span}

\DeclareMathOperator{\Rank}{rank}

\DeclareMathOperator{\Card}{card}

\DeclareMathOperator{\Dep}{Dep}
\DeclareMathOperator{\Depn}{DepN}
\DeclareMathOperator{\Ker}{Ker}
\DeclareMathOperator{\mult}{mult}
\DeclareMathOperator{\col}{col}
\DeclareMathOperator{\row}{row}

\DeclareMathOperator{\coef}{coef}

\makeatletter
\def\widebreve{\mathpalette\wide@breve}
\def\wide@breve#1#2{\sbox\z@{$#1#2$}%
	\mathop{\vbox{\m@th\ialign{##\crcr
				\Kern0.08em\brevefill#1{0.8\wd\z@}\crcr\noalign{\nointerlineskip}%
				$\hss#1#2\hss$\crcr}}}\limits}
\def\brevefill#1#2{$\m@th\sbox\tw@{$#1($}%
	\hss\resizebox{#2}{\wd\tw@}{\rotatebox[origin=c]{90}{\upshape(}}\hss$}
\makeatletter

\newcommand{\RR}{\mathbb R}

\newcommand{\NN}{\mathbb N}
\newcommand{\ZZ}{\mathbb Z}

\newcommand{\cT}{\mathcal T}

\newcommand{\Sym}{\mathbb{S}}

\newcommand{\cS}{\mathcal{S}}

\newcommand{\cM}{\mathcal M}

\newcommand{\cH}{\mathcal H}

\newcommand{\cZ}{\mathcal Z}
\newcommand{\cC}{\mathcal C}

\newcommand{\benu}{\begin{enumerate}}
	\newcommand{\eenu}{\end{enumerate}}
\newcommand{\bop}{\begin{opomba}}
	\newcommand{\eop}{\end{opomba}}

\newcommand{\Bor}{\mathrm{Bor}}

\newcommand{\tr}{\mathrm{tr}}
\newcommand{\supp}{\mathrm{supp}}

\newtheorem{theorem}{Theorem}[section]
\newtheorem{corollary}[theorem]{Corollary}
\newtheorem{lemma}[theorem]{Lemma}
\newtheorem{proposition}[theorem]{Proposition}

\theoremstyle{definition}

\definecolor{green-new}{rgb}{0.0, 0.5, 0.0}

\definecolor{cyan}{rgb}{0.0, 0.8, 1.0}

\theoremstyle{definition}

\newtheorem{remark}[theorem]{Remark}

\numberwithin{equation}{section}

\begin{document}

	\numberwithin{equation}{section}

	\title[]
	{Matricial Gaussian quadrature rules: singular case}
	
	\author[A. Zalar]{Alja\v z Zalar${}^{1}$}
	\address{Alja\v z Zalar, 
		Faculty of Computer and Information Science, University of Ljubljana  \& 
		Faculty of Mathematics and Physics, University of Ljubljana  \&
		Institute of Mathematics, Physics and Mechanics, Ljubljana, Slovenia.}
	\email{aljaz.zalar@fri.uni-lj.si}
	\thanks{${}^1$Supported by the ARIS (Slovenian Research and Innovation Agency)
		research core funding No.\ P1-0288 and grants No.\ J1-50002, J1-60011.}
	
	\author[I. Zobovi\v c]{Igor Zobovi\v c${}^{2}$}
	\address{Igor Zobovi\v c, 
		Institute of Mathematics, Physics and Mechanics, Ljubljana, Slovenia.}
	\email{igor.zobovic@imfm.si}
	\thanks{${}^2$Supported by the ARIS (Slovenian Research and Innovation Agency)
		research core funding No.\ P1-0288.}
	
	\begin{abstract}
		Let $L$ be a linear operator on univariate polynomials of bounded degree taking values in real symmetric matrices, whose moment matrix is positive semidefinite. Assume that $L$ admits a positive matrix-valued representing measure $\mu$. Any finitely atomic representing measure with the smallest sum of the ranks of the matricial masses is called minimal. In this paper, we characterize the existence of a minimal representing measure that contains a prescribed atom with a prescribed rank of the corresponding mass, thereby generalizing our recent result \cite{ZZ25}, which addresses the same problem in the case where the moment matrix is positive definite. As a corollary, we obtain a constructive, linear-algebraic proof of the strong truncated Hamburger matrix moment problem.
		\looseness=-1
	\end{abstract}

	
	\date{\today}
	\maketitle
	
	
	\section{Introduction}
	\label{introduction}    
	In this paper we study matricial Gaussian quadrature rules for a linear operator $L$ on univariate polynomials of bounded degree with values in real symmetric matrices.
	We work in the setting where the truncated moment matrix (see \eqref{def:moment-matrix}) is positive semidefinite and 
	$L$ admits a positive matrix-valued representing measure. 
	More precisely, for a fixed real number 
	$t$ and an integer $m\in\NN\cup\{0\}$, we characterize the conditions under which there exists a minimal representing measure for $L$ that contains $t$ in its support and whose corresponding mass has rank $m$.
	Our results extend the positive definite case treated in \cite{ZZ25} to the singular setting.
	These results will be important in solving the truncated univariate matrix rational moment problem, in analogy with the scalar case studied in \cite{NZ25}. 
	We present the positive semidefinite case separately, since—unlike in the positive definite situation—the moment matrix exhibits nontrivial column relations that, due to its recursively generated structure, propagate to neighboring block columns on the right. This propagation significantly increases the technical complexity of both the construction and the proofs.
	
	{Let $i,j\in \ZZ$, $i\leq j$. We write $[i;j]:=\{i,i+1,\ldots,j\}$.
		For $j\in \NN$, $[j]:=[1;j]=\{1,2,\ldots,j\}$ is used for short.}
	Let $k\in \NN\cup \{0\}$ and $p\in \NN$.
	We denote by $\RR[x]_{\leq k}$ the vector space of univariate polynomials of degree at most $k$ and by $\Sym_p(\RR)$ the set of real symmetric matrices of size $p\times p$.
	For a given linear operator
	\begin{equation} 
		\label{def:linear-mapping}
		L:\RR[x]_{\leq 2n}\to \Sym_p(\RR)
		\        \end{equation}
	denote its \textbf{matricial moments} by $S_i:=L(x^{i})$ for $i\in [0;2n]$. 
	Assume that 
	$L$ admits 
	a positive $\Sym_p(\RR)$-valued measure $\mu$ (see Subsection \ref{subsec:matrix-measure}), i.e., 
	\begin{equation}
		\label{moment-measure-cond}
		L(p)=\int_{\RR} p\; d\mu\quad \text{for every }p\in \RR[x]_{\leq 2n}.
	\end{equation}
	Every measure $\mu$ satisfying \eqref{moment-measure-cond}
	is called a \textbf{representing measure for $L$}. By \cite[Theorem 2.7.6]{BW11}, whenever $L$ has a representing measure, there is one of the form $\mu=\sum_{j=1}^\ell \delta_{x_j}A_j$, where each $0\neq A_j\in \Sym_p(\RR)$ is positive semidefinite, $\sum_{j=1}^\ell \Rank A_j = \Rank M(n)$ with $M(n)$ being a truncated moment matrix of $(L(x^i))_{i\in [0;2n]}$ (see \eqref{def:moment-matrix})
	and $\delta_{x_j}$ stands for the Dirac measure supported in $x_j$. Such a measure is called \textbf{minimal} because the total sum of the ranks of its matricial masses is minimal among all representing measures for $L$.
	In this case \eqref{moment-measure-cond} is equal to
	\begin{equation}
		\label{moment-measure-cond-Gaussian}
		L(p)=\sum_{j=1}^\ell p(x_j)A_j \quad \text{for every }p\in \RR[x]_{\leq 2n},
	\end{equation}
	and \eqref{moment-measure-cond-Gaussian} is called 
	a \textbf{matricial Gaussian quadrature rule for $L$}.
	The points $x_j$ are called \textbf{atoms} of the measure $\mu$. If $x_1,\ldots,x_\ell$ are pairwise distinct, 
	then for each $j$, the matrix $A_j = \mu(\{x_j\})$ is called the \textbf{mass} of $\mu$ at $x_j$ and its rank is called the \textbf{multiplicity} of $x_j$ in $\mu$, which we denote by 
	$\mult_\mu x_j$. If $x$ is not an atom of $\mu$, then $\mult_\mu x:=0$.\\
	
	The central problem of this paper is the following:\\
	
	\noindent \textbf{Problem.}
	Let $L$ be as in \eqref{def:linear-mapping}
	and assume that it admits a positive matrix-valued representing measure.
	Given $t\in \RR$ and $m\in \NN\cup \{0\}$,
	characterize the existence of a minimal representing measure $\mu$ for $L$ such that $\mult_\mu t=m$.\\
	
	{
		The main result of the paper is the solution to the Problem above.
		
		\begin{theorem}
			\label{th:mainTheorem}
			Let $n, p \in \NN$
			and
			$L:\RR[x]_{\leq 2n}\to \Sym_p(\RR)$
			be a linear operator,
			which admits a representing measure. Fix $t \in \RR$ and $m\in \NN\cup\{0\}$.
			Let $T_i:=L((x-t)^{i})$ for $i\in [0;2n]$ and
			$$\cM_1:=(T_{i+j-2})_{i,j=1}^{n+1},\qquad 
			\cM_2:=(T_{i+j-2})_{\substack{i=1,\ldots,n,\\j=1,\ldots,n+2}}.$$
			Denote by $\mathbf{y}_j$ the $j$--th column of $\cM_1$
			and by $\mathbf{u}_j$ the $j$--th column of $\cM_2$.
			Let 
			\begin{align*} 
				A:=
				\Big\{i\in [p]\colon 
				&\text{There exists }k\in [n]\text{ such that }
				\mathbf{y}_{kp+i}=\sum_{\ell=p+1}^{kp+i-1}\gamma_\ell \mathbf{y}_\ell \text{ for some }\gamma_\ell\in \RR
				\text{ and }\\
				&\mathbf{y}_{(k-1)p+i}\neq \sum_{\ell=1}^{(k-1)p+i-1}\delta_\ell \mathbf{y}_\ell \text{ for all }\delta_\ell\in \RR
				\Big\}
			\end{align*}
			and
			\begin{align*} 
				B:=
				\Big\{i\in [p]\colon 
				&
				\mathbf{u}_{(n+1)p+i}=\sum_{\ell=p+1}^{(n+1)p+i-1}\gamma_\ell \mathbf{u}_\ell\text{ for some }\gamma_\ell\in \RR
				\text{ and }\\
				&{\mathbf{y}_{np+i}\neq \sum_{\ell=1}^{np+i-1}\delta_\ell \mathbf{y}_\ell \text{ for all }\delta_\ell\in \RR}
				\Big\}.
			\end{align*}
			Then the following statements are equivalent:
			\begin{enumerate}
				\item 
				\label{th:mainTheorem-pt1}
				There exists a minimal
				representing matrix measure $\mu$
				for $L$ such that $\mult_\mu t=m$.
				\item 
				\label{th:mainTheorem-pt2}
				$\Card A\leq m\leq \Card(A\cup B)$.
			\end{enumerate}
			
			{
				Moreover, writing 
				$B:=\{i_1,i_2,\ldots,i_{\Card B}\}$
				and
				$\{j_1,j_2,\ldots,j_{p-\Card B}\}:=[p]\setminus B$,
				the representing measure is unique if and only if
				\begin{align}
					\label{cond:uniqueness}
					\begin{split}
						&B=\Big\{i\in [p]\colon 
						\mathbf{y}_{np+i}\neq \sum_{\ell=1}^{np+i-1}\delta_\ell \mathbf{y}_\ell \text{ for all }\delta_\ell\in \RR
						\Big\},\\
						&
						m=\Card(A\cup B),\\
						&\Ker(T_{i+j-2})_{\substack{i=1,\ldots,n,\\
								j=2,\ldots,n+1}}
						\subseteq 
						\Ker 
						(
						\mathbf{u}_{(n+1)p+i_1}\;
						\mathbf{u}_{(n+1)p+i_2}\;
						\ldots\;
						\mathbf{u}_{(n+1)p+i_{\Card B}}
						)^T,\\
						&\Ker(T_{i+j-2})_{\substack{i=1,\ldots,n,\\
								j=2,\ldots,n+1}}
						\subseteq 
						\Ker 
						(
						\mathbf{u}_{(n+1)p+j_1}\;
						\mathbf{u}_{(n+1)p+j_2}\;
						\ldots\;
						\mathbf{u}_{(n+1)p+j_{p-\Card B}}
						)^T. 
					\end{split}
				\end{align}
			}
		\end{theorem}
		
		{
			\begin{remark}
				A special case of Theorem~\ref{th:mainTheorem}, corresponding to the situation where $\cM_1$ is positive definite, was established in \cite[Theorem~1.1]{ZZ25}.  
				Under this assumption we have $A=\emptyset$, and  
				\begin{align}
					\label{card-B}
					\begin{split}
						\Card B 
						&= 
						\Rank (T_{i+j-2})_{\substack{i=1,\ldots,n \\ j=2,\ldots,n+1}} - (n-1)p\\
						&=
						\Rank (S_{i+j-1}-tS_{i+j-2})_{i,j=1}^n- (n-1)p,
					\end{split}
				\end{align}
				where the second equality follows by Proposition 
				\ref{linear transform invariance-nc}.\eqref{linear transform invariance-nc-1}, i.e., in the notation of Proposition \ref{linear transform invariance-nc} we have
				$$
				(J_\phi\otimes I_p)^{-T}
				(T_{i+j-2})_{\substack{i=1,\ldots,n \\ j=2,\ldots,n+1}}
				(J_\phi\otimes I_p)^{-1}
				=
				(S_{i+j-1}-tS_{i+j-2})_{i,j=1}^n,
				$$
				where $J_\phi:\RR^n\to \RR^n$.
				Hence,
				$$\Rank (T_{i+j-2})_{\substack{i=1,\ldots,n \\ j=2,\ldots,n+1}}=
				\Rank (S_{i+j-1}-tS_{i+j-2})_{i,j=1}^n.$$
				Using \eqref{card-B}, Theorem~\ref{th:mainTheorem}.\eqref{th:mainTheorem-pt2}  coincides with the statement of \cite[Theorem~1.1.(2)]{ZZ25}
				in the case $\cM_1$ is positive definite.
				We also note that, in order to state \cite[Theorem~1.1]{ZZ25}, it was sufficient to use the localizing matrix  
				$
				(S_{i+j-1} - tS_{i+j-2})_{i,j=1}^n,
				$
				since $\cM_1$ was assumed to be invertible.  
				In the singular case, however, it is more convenient to work with the matrices $\cM_1$ and $\cM_2$, as this makes it possible to capture column dependencies within $\cM_1$ as well. This dependencies are essential for the formulation of Theorem~\ref{th:mainTheorem}.
			\end{remark}
		}
		
		In the case $m=0$, we can add another equivalence to Theorem \ref{th:mainTheorem}.
		
		\begin{theorem}
			\label{co:corollary-atom-avoidance}
			Let $n, p \in \NN$
			and
			$L:\RR[x]_{\leq 2n}\to \Sym_p(\RR)$
			be a linear operator,
			which admits a representing measure. Fix $t \in \RR$.
			Let $T_i:=L((x-t)^{i})$ for $i\in [0;2n]$, 
			\begin{equation}
				\label{051025-1917}
				\cM_1:=(T_{i+j-2})_{i,j=1}^{n+1},\quad
				\cM_2:=(T_{i+j-2})_{i,j=1}^{n}
				\quad\text{and}\quad
				\cM_3:=(T_{i+j})_{i,j=1}^{n}.
			\end{equation}
			Denote by $\mathbf{y}_j$ the $j$--th column of $\cM_1$.
			Let 
			\begin{align*} 
				A:=
				\Big\{i\in [p]\colon 
				&\text{There exists }k\in [n]\text{ such that }
				\mathbf{y}_{kp+i}=\sum_{\ell=p+1}^{kp+i-1}\gamma_\ell \mathbf{y}_\ell \text{ for some }\gamma_\ell\in \RR
				\text{ and }\\
				&\mathbf{y}_{(k-1)p+i}\neq \sum_{\ell=1}^{(k-1)p+i-1}\delta_\ell \mathbf{y}_\ell \text{ for all }\delta_\ell\in \RR
				\Big\}.
			\end{align*}
			The following statements are equivalent:
			\begin{enumerate} 
				\item 
				\label{co:corollary-atom-avoidance-pt1}
				There exists a minimal 
				representing matrix measure $\mu$
				for $L$ such that $\mult_\mu t=0$.
				\item 
				\label{co:corollary-atom-avoidance-pt2}
				$A=\emptyset$.
				\item 
				\label{co:corollary-atom-avoidance-pt3}
				$\Ker \cM_2=\Ker \cM_3$.
			\end{enumerate}
		\end{theorem}
		
		Using Theorem \ref{co:corollary-atom-avoidance} we obtain \cite[Theorem 3.3]{Sim06}, which is a solution
		to a strong truncated matrix Hamburger moment problem.
		
		\begin{corollary}
			\label{co:Simonov-psd}
			Let $n_1,n_2, p \in \NN$ and
			$S_i\in \Sym_p(\RR)$ for $i\in [-2n_1;2n_2]$. Then the following statements are equivalent:
			\begin{enumerate} 
				\item 
				\label{co:Simonov-psd-pt00}
				There exists a measure $\mu$
				such that $S_i=\int_\RR x^i d\mu$ for $i\in [-2n_1;2n_2]$.
				\item 
				\label{co:Simonov-psd-pt0}
				There exists a finitely atomic measure $\mu$
				such that $S_i=\int_\RR x^i d\mu$ for $i\in [-2n_1;2n_2]$.
				\item 
				\label{co:Simonov-psd-pt1}
				There exists a minimal representing measure $\mu$
				such that 
				$$S_i=\int_\RR x^i d\mu \quad \text{for } i\in [-2n_1;2n_2].$$
				\item
				\label{co:Simonov-psd-pt3}
				$(S_{i+j})_{i,j=-n_1}^{n_2}$ is positive semidefinite
				and $\Ker (S_{i+j})_{i,j=-n_1}^{n_2-1}=\Ker (S_{i+j})_{i,j=-n_1+1}^{n_2}$.
				\item 
				\label{co:Simonov-psd-pt2}
				$(S_{i+j})_{i,j=-n_1}^{n_2}$ is positive semidefinite
				and for every sequence $\{v_k\}_{k=0}^{n_1+n_2-1} \subset \RR^p$
				\begin{align*}
					\sum_{i,j=-n_1}^{n_2-1}v_{j+n_1}^TS_{i+j}v_{i+n_1} = 0 \quad \text{if and only if} \quad \sum_{i,j=-n_1}^{n_2-1}v_{j+n_1}^TS_{i+j+2}v_{i+n_1} = 0.
				\end{align*}
			\end{enumerate}
		\end{corollary}
		
		{
			In \cite[Theorem 1.4]{BKRSV20}, the authors solved the scalar version  of the Problem in the nonsingular case (i.e., $p=1$ in \eqref{def:linear-mapping} and the moment matrix $M(n)$ as in \eqref{def:moment-matrix} is positive definite) in terms of symmetric determinantal representations involving moment matrices using convex analysis and algebraic geometry.
			For an alternative proof using moment theory, and an extension to finitely many prescribed atoms, see \cite{NZ+}.
			In \cite[Theorem 1.1]{ZZ25}, we extended \cite[Theorem 1.4]{BKRSV20} to the general matrix case (i.e., $p\in \NN$). 
			We remark that, in the scalar setting, the restriction to positive definite moment matrices \( M(n) \) is natural. Specifically, when \( M(n) \) is positive semidefinite but not positive definite, the minimal representing measure is uniquely determined \cite[Theorems~3.9 and~3.10]{CF91}. This uniqueness property does not extend to the matrix-valued case, whence a formulation of the problem with positive semidefinite \( M(n) \) remains relevant. Furthermore, this variant of the problem is technically more intricate and is addressed in the present paper.
			
			In \cite{FKM24}, the authors described, for each \( t \in \mathbb{R} \), the set of all possible masses at \( t \) among all representing measures for \( L \), addressing a question related to the Problem above. The multivariate analogue of this problem was investigated in \cite{MS23+, MS23++}. The main distinction of our work lies in the fact that we focus on \textit{minimal} representing measures with a \textit{fixed rank} of the mass at \( t \).
			
			The works \cite{DD02, DLR96, DS03} focus on the computation of atoms and masses in matricial Gaussian quadrature rules, where the moment \( S_{2n+1} \) is fixed, and consequently, the minimal representing measure is unique. The corresponding formulas are expressed in terms of the zeros of the associated orthogonal matrix polynomial. A distinctive feature of our results is that we do not fix \( S_{2n+1} \) \textit{a priori}. Instead, we characterize the existence of a suitable \( S_{2n+1} \) that yields a minimal representing measure containing a prescribed atom with a prescribed multiplicity. In the proof, we explicitly construct such an \( S_{2n+1} \) with the required properties, ensuring that the extended moment matrix \( M(n+1) \) satisfies \( \Rank M(n+1) = \Rank M(n) \) and admits an appropriate block column relation (see Subsection~\ref{subsec:evaluation}).
			
			The techniques employed in \cite{Sim06} to establish Theorem~\ref{co:Simonov-psd} rely on advanced operator-theoretic methods, involving the analysis of self-adjoint extensions of a certain linear operator that is not necessarily defined on the entire finite-dimensional Hilbert space of vector-valued Laurent polynomials. In contrast, the present paper provides a constructive, linear-algebraic proof, in the sense that the representing measures can be obtained explicitly by following the steps outlined in the proof of Theorem~\ref{th:mainTheorem}.
			
		}
		
		\subsection{Reader's guide} 
		In Section \ref{sec:prel} we introduce the notation and some preliminary results.
		In Section \ref{sec:proofOfMainTheorem} we prove
		the implication $\eqref{th:mainTheorem-pt2} \Rightarrow \eqref{th:mainTheorem-pt1}$ of Theorem \ref{th:mainTheorem}.
		In the last paragraph of the section we also prove the moreover part of Theorem \ref{th:mainTheorem}.
		In Section \ref{section:proof-main-reverse} 
		we prove the remaining implication $\eqref{th:mainTheorem-pt1} \Rightarrow \eqref{th:mainTheorem-pt2}$.
		Section \ref{section:proof-theorem-multiplicity-0} is devoted to the proof of Theorem \ref{co:corollary-atom-avoidance},
		while Section \ref{section:corollary} to Corollary \ref{co:Simonov-psd}.
		Finally, in Section \ref{sec:example} 
		we demonstrate the statement of Theorem \ref{th:mainTheorem}
		on a numerical example.

	\section{Preliminaries}
	\label{sec:prel}
	
	Let $m,m_1,m_2\in \NN$.
	We write $M_{m_1\times m_2}(\RR)$ for the set of $m_1\times m_2$ real matrices
	and $M_m(\RR)\equiv M_{m\times m}(\RR)$ for short.
	For a matrix $A\in M_{m_1\times m_2}(\RR)$
	we call the linear span of its columns
	a $\textbf{column space}$ and denote it by $\cC(A)$.
	We denote by $I_{m}$ the identity $m\times m$ matrix and by $\mathbf{0}_{m_1\times m_2}$ the zero
	$m_1\times m_2$ matrix, while $\mathbf{0}_m\equiv \mathbf{0}_{m\times m}$ for short.
	We use $M_m(\RR[x])$ to denote $m\times m$ matrices over $\RR[x]$. The elements of $M_m(\RR[x])$ 
	are called \textbf{matrix polynomials}.
	
	Let $p\in \NN$. 
	For $A\in \Sym_p(\RR)$ the notation $A\succeq 0$ (resp.\ $A\succ 0$) means $A$ is positive semidefinite (psd) (resp.\ positive definite (pd)).
	We use $\Sym^{\succeq 0}_p(\RR)$ for the subset  of all psd matrices in $\Sym_p(\RR)$.
	
	Given a polynomial $p(x)\in \RR[x]$, we write 
	$\cZ(p(x)):=\{x\in \RR\colon p(x)=0\}$ for the set of its zeros
	{and by $\cZ_{\RR}(p(x)):=\cZ(p(x))\cap\RR$ the set of its real zeros. For $a\in \cZ_{\RR}(p(x))$ we denote by $\mult_{p(x)} a$ the vanishing order of $a$ as a zero of $p(x)$.
		
		{Let $A_k$, $k\in [i;j]$ be given matrices. We use \\
			$$\row(A_k)_{k\in [i;j]}\equiv \row(A_i,A_{i+1},\ldots,A_j):=\begin{pmatrix}
				A_i & A_{i+1} & \cdots & A_j
			\end{pmatrix}$$ \\
			and
			$$
			\col(A_k)_{k\in [i;j]}\equiv \col(A_i,A_{i+1},\ldots,A_j)
			:=
			\begin{pmatrix}
				A_i \\ A_{i+1} \\ \vdots \\ A_j
			\end{pmatrix}
			$$ \\
			for the row vector and the column vector with entries $A_k$,
			respectively.}

	\subsection{Matrix measures}
	\label{subsec:matrix-measure}
	Let $\Bor(\RR)$
	be the Borel $\sigma$-algebra of $\RR$. 
	We call 
	$$\mu=(\mu_{ij})_{i,j=1}^p:\Bor(\RR)\to \Sym_p(\RR)$$ 
	a $p\times p$ \textbf{Borel matrix-valued measure} supported on $\RR$
	(or \textbf{positive $\Sym_p(\RR)$-valued measure})
	if
	\begin{enumerate}
		\item 
		$\mu_{ij}:\Bor(\RR)\to \RR$ 
		is a real measure for every $i,j\in [p]$ and
		\smallskip
		\item
		$\mu(\Delta)\succeq 0$ for every $\Delta\in \Bor(\RR)$.
	\end{enumerate} 
	
	A positive $\Sym_p(\RR)$-valued measure $\mu$ is \textbf{finitely atomic} if there exists a finite set $M\in \Bor(\RR)$
	such that
	$\mu(\RR\setminus M)=\mathbf 0_{p}$
	or
	equivalently,
	$\mu=\sum_{j=1}^\ell \delta_{x_j}A_j$
	for some $\ell\in \NN$, $x_j\in \RR$, $A_j\in \Sym_p^{\succeq 0}(\RR)$.
	We say $\mu$ is \textbf{$r$--atomic} if $\sum_{j=1}^{\ell}\Rank A_j = r$.
	
	Let $\mu$ be a positive $\Sym_p(\RR)$-valued measure
	and 
	$\tau:=\tr(\mu)=\sum_{i=1}^p \mu_{ii}$ denote its trace measure. 
	A polynomial $f\in \RR[x]$ is $\mu$-integrable if 
	$f\in L^1(\tau)$. 
	We define its integral by
	$$
	\int_\RR f\;d\mu
	:=
	\Big(\int_\RR f\; d\mu_{ij}\Big)_{i,j=1}^p.
	$$

	\subsection{Riesz mapping}
	\label{subsec:Riesz}
	One can define $L$ as in \eqref{def:linear-mapping}
	by a sequence of its values on monomials $x^i$, $i\in [0;2n]$. 
	Throughout the paper we will denote these values
	by $S_i:=L(x^i)$. 
	If
	$\mathcal S:=(S_0,S_1,\ldots,S_{2n})\in (\Sym_p(\RR))^{2n+1}$
	is given, then we denote the corresponding linear mapping on 
	$\RR[x]_{\leq 2n}$ by $L_{\mathcal S}$ and call it a \textbf{Riesz 
		mapping of $\mathcal S$}. 

	\subsection{Moment matrix and its column dependencies}
	For $n\in \NN$
	and 
	\begin{equation}
		\label{def:sequence}
		\mathcal S\equiv \mathcal S^{(2n)}:=(S_0,S_1,\ldots,S_{2n})\in (\Sym_p(\RR))^{2n+1},
	\end{equation}
	we denote by
	\begin{equation}
		\label{def:moment-matrix}
		M(n)\equiv M_\cS(n):= \begin{pmatrix}S_{i+j-2}\end{pmatrix}_{i,j = 1}^{n+1} =
		\kbordermatrix{
			& \mathit{1} & X & X^2 & \cdots & X^n\\
			\mathit{1} & S_0 & S_1 & S_2 & \cdots & S_n\\[0.2em]
			X & S_1 & S_2 & \iddots & \iddots & S_{n+1}\\[0.2em]
			X^2 & S_2 & \iddots & \iddots & \iddots & \vdots\\[0.2em]
			\vdots & \vdots 	& \iddots & \iddots & \iddots & S_{2n-1}\\[0.2em]
			X^n & S_n & S_{n+1} & \cdots & S_{2n-1} & S_{2n}
		}
	\end{equation}
	the corresponding 
	\textbf{$n$--th truncated moment matrix.}
	We write 
	\begin{equation} 
		\label{def:columns-mm}
		X^i := \col(S_{i+\ell})_{\ell\in [0;n]}.
	\end{equation}
	
	\begin{remark}
	Further, we will also use $X^i_{\cS}$ to denote the $i$--th block column of the largest moment matrix, fully determined by $\cS$.
	\end{remark}
	In the notation \eqref{def:columns-mm}, we have 
	$M(n)=\row(X^i)_{i\in [0;n]}$.
	Let us denote the columns of each block $X^i$ by $x_j^{(i)}$, i.e.,
	\begin{equation}
	\label{columns-of-Yi}
	X^i = \begin{pmatrix}
		x_1^{(i)} & x_2^{(i)} & \cdots & x_p^{(i)}    
	\end{pmatrix} = \row(x^{(i)}_j)_{j\in [p]}.
	\end{equation}
	
	For $i\in [0;n]$ let 
	$\Dep(X^i)\equiv \Dep(X_{\cS}^i)$
	be the set of all indices $j\in [p]$ for which the column $x_j^{(i)}$ is in the span of the previous columns of the matrix $M(n)$, i.e.,
	\begin{align}
	\label{eq:defDep-of-Y-i}
	\Dep(X^i) 
	&:= \left \{ j \in [p] \colon 
	x^{(i)}_j\in
	\cC\begin{pmatrix}
		\row(X^k)_{k\in [0;i-1]}
		&
		\row(x_k^{(i)})_{k\in [j-1]}
	\end{pmatrix}
	\right\}.
	\end{align}
	Similarly, for $i\in [n]$ let
	$\Dep_1(X^i)\equiv \Dep_1(X_{\cS}^i)$
	be the set of all indices $j\in [p]$ for which the column $x_j^{(i)}$ is in the span of all previous columns of the matrix $M(n)$ \textit{starting from the block column} $X^1$, i.e.,
	\begin{align}
	\label{eq:defDep-of-Y-i-alternative}
	\Dep_1(X^i) 
	&:= \left \{ j \in [p] \colon 
	x^{(i)}_j\in
	\cC\begin{pmatrix}
		\row(X^k)_{k\in [i-1]}
		&
		\row(x_k^{(i)})_{k\in [j-1]}
	\end{pmatrix}
	\right\}.
	\end{align}
	Define $\Dep_1(X^0) := \emptyset$. Note that $\Dep_1(X^i) \subseteq \Dep(X^i)$ for each $i\in [0;n]$.
	{
	Further, let 
	\begin{align}
		\label{eq:defDepnew-of-Y-i}
		\begin{split}
			\Depn(X^{0}) &:= 	\Dep(X^0), \\
			\Depn(X^i) &:= \Dep(X^i) \setminus \Dep(X^{i-1}) \quad \text{for} \quad i\in [n],\\
			\Depn_1(X^0) &:= \emptyset,\\
			\Depn_1(X^i) &:= \Dep_1(X^i) \setminus \Dep(X^{i-1}) \quad \text{for} \quad i\in [n],\\
			\Depn_0(X^i) &:= \Depn(X^i) \setminus \Depn_1(X^i) \quad \text{for} \quad i\in [0;n].
		\end{split}
	\end{align}

	\subsection{Evaluation of 
		a matrix polynomial on $M(n)$}
	\label{subsec:evaluation}
	Let $\cS$ be a sequence as in \eqref{def:sequence}
	and $M(n)$ its corresponding moment matrix.
	Given a matrix polynomial 
	\begin{equation} 
		\label{def:matrix-poly}
		P(x)=\sum_{i=0}^n x^i P_i \in M_p(\RR[x]),
	\end{equation}
	the \textbf{evaluation} $P(X)$
	on $M(n)$ is a matrix, obtained by 
	replacing each monomial of $P$ by the corresponding column of $M(n)$ and multiplying it 
	with the matrix coefficients $P_i$ from the right, i.e., 
	$$P(X) \equiv P(X_\cS)
	:=
	\sum_{i=0}^n X_{\cS}^i P_i
	{=
		X_{\cS}^0 P_0+
		X_{\cS}^1 P_1+
		\cdots+
		X_{\cS}^n P_n
	}
	\in M_{(n+1)p\times p}(\RR),$$
	where $X^i$
	are as in \eqref{def:columns-mm} above.
	If $P(X)=\mathbf{0}_{(n+1)p\times p}$, then we say $P(x)$ is a \textbf{block column relation} of $M(n)$.
	
	Given a matrix polynomial $P(x)$
	as in \eqref{def:matrix-poly}, we write 
	\begin{equation}
		\label{def:coefficients}
		\coef(P(x)):=\col(P_i)_{i\in [0;n]}
	\end{equation} 
	to denote the vector of its matrix coefficients.

	A moment matrix $M(n)$ is 
	\textbf{block--recursively generated} 
	if for every block column relation 
	$P(x) = \sum_{i=0}^{j}x^{i}P_{i} \in M_p(\RR[x])$, $0\leq j<n$, of $M(n)$,
	the matrix polynomial $(xP)(x)=\sum_{i=0}^{j}x^{i+1}P_i$ is also a block column relation of $M(n)$.
	

	\subsection{Solution to the truncated matricial Hamburger moment problem}
	
	\begin{theorem}
		[{\cite[Theorem 2.7.6]{BW11}}]
		\label{th:Hamburger-matricial}
		Let $n, p \in \NN$ and let $\mathcal S:=(S_0,S_1,\ldots,S_{2n})\in (\Sym_p(\RR))^{2n+1}$
		be a sequence with a moment matrix $M(n)$.
		Then the following statements are equivalent:
		\begin{enumerate}
			\item 
			There exists a  
			representing matrix measure for $\mathcal S$.
			\item 
			There exists a $(\Rank M(n))$--atomic 
			representing matrix measure for $\mathcal S$.
			\item 
			$M(n)$
			is positive semidefinite
			and 
			$
			\cC\big(
			\col(S_{n+i})_{i\in [n]}
			\big)
			\subseteq
			\cC(M(n-1)).
			$
		\end{enumerate}
	\end{theorem}
	
	\begin{remark}
		The truncated matrix Hamburger moment problem was also considered in 
		\cite{Bol96,Dym89,DFKMT09}.
	\end{remark}

	The following lemma connects the support $\supp(\mu)$
	of a representing measure $\mu$ for 
	$\cS$ as in Theorem \ref{th:Hamburger-matricial}
	with block column relations of $M(n)$.
	
	\begin{lemma}
		\label{lemma:atoms}
		Let $n, p \in \NN$ and let $\mathcal S:=(S_0,S_1,\ldots,S_{2n})\in (\Sym_p(\RR))^{2n+1}$ be a sequence with a moment matrix $M(n)$ and a representing measure $\mu$.
		Assume that $H(x)=\sum_{i=0}^n x^i H_i\in M_p(\RR[x])$ is a block column relation of $M(n)$ such that $H_n$
		is invertible. Then
		\begin{equation}
			\label{measure-support-inclusion}
			\supp(\mu) \subseteq \cZ(\det H(x))
			\quad
			\text{and}
			\quad
			\mult_\mu \xi\leq \mult_{\det H(x)} \xi
			\;\; \text{for each }\xi\in \RR.
		\end{equation}
	\end{lemma}
	
	\begin{proof}
		The inclusion 
		$\supp(\mu) \subseteq \cZ(\det H(x))$
		is \cite[Lemma 5.53]{KT22}. It remains to prove 
		the second assertion in \eqref{measure-support-inclusion}.
		Let $L\equiv L_\cS:\RR[x]_{\leq 2n}\to \Sym_p(\RR)$ be the Riesz mapping of $\cS$.  
		We define a functional 
		$$
		\Lambda_L:\Sym_p(\RR[x])\to \RR,\quad
		\Lambda_L((F_{jk})_{j,k}):=
		\sum_{1\leq j\leq p}
		L(F_{jj})_{j,j}+
		\sum_{1\leq j< k \leq p}
		2L(F_{jk})_{j,k}
		$$
		By the real version of \cite[Proposition 4.5]{MS23+},
		representing measures for $L$ coincide with representing measures for $\Lambda_L$. Here,
		$\widetilde \mu$ is a representing measure for $\Lambda_L$
		if and only if 
		$\Lambda_L(F)=\int_\RR \tr(F(x)\Phi(x))d\widetilde \tau$,
		where $\widetilde\tau=\tr(\widetilde{\mu})$ is a trace measure and $\Phi(x)$ is so-called Radon-Nikodym matrix of $\widetilde\mu$ (see \cite[Section 1]{Sch87}).
		Let $\coef(H(x))$ be as in \eqref{def:coefficients}.
		Since $H(x)$ is a block column relation of $M(n)$, in particular we have that
		$$\Lambda_L(H(x)(H(x))^T)=
		\tr
		\big(\coef(H(x))^TM(n)\coef(H(x))\big)=0,
		$$
		where the first equality follows by a short computation as in \cite[Lemma 3.2]{MS23++}.
		By the real version of \cite[Lemma 7.7]{MS23+} (using $P(x)$ being constantly equal to the identity function on $\RR$), in particular it follows that
		$$
		\Rank\mu(\{\xi\})
		\leq 
		\dim \Ker(H(\xi)H(\xi)^T)
		=\dim \Ker H(\xi)
		$$
		for every $\xi\in \supp(\mu)$.
		To conclude the proof we need to show that 
		\begin{equation}
			\label{remaining-to-prove-2210}
			\dim \Ker H(\xi)\leq \mult_{\det H(x)}\xi.
		\end{equation}
		By \cite[Theorem 1.1]{GLR82}, 
		$I_{np}x-C=E(x)(H_n^{-1}H(x)\oplus I_{(n-1)p})F(x)$,\
		where $C$ is a $np\times np$ companion matrix of $H_n^{-1}H(x)$ and $E(x)$, $F(x)$ are $np\times np$
		matrix polynomials with constant nonzero determinant.
		Hence, $\dim \Ker(I_{np}\xi-C)=\dim \Ker H(\xi)$. 
		Since 
		the algebraic multiplicity of $\xi$ as the eigenvalue of the matrix $C$ is not smaller than the dimension of the eigenspace of $C$ corresponding to $\xi$, 
		\eqref{remaining-to-prove-2210} follows.
	\end{proof}

	\begin{proposition}
		\label{recursive-generation-if-measure}
		Let $n, p \in \NN$ and let $\mathcal S:=(S_0,S_1,\ldots,S_{2n})\in (\Sym_p(\RR))^{2n+1}$ be a sequence with a representing measure $\mu$. Then $M(n)$ is 
		block--recursively generated.
	\end{proposition}
	
	\begin{proof}
		By the matricial Richter-Tchakaloff's theorem \cite[Theorem 5.1]{MS23+},
		there exists a finitely atomic representing measure $\widetilde\mu$ for $\cS$.
		Let $\cS^{(\infty)}=(S_i)_{i\in \NN_0}$ be an infinite sequence, obtained
		by $S_i:=\int_\RR x^i d\widetilde\mu$. 
		Let $M(\infty)=(S_{i+j-2})_{i,j\in \NN}$ be the corresponding Hankel matrix.
		Let $P(x)=\sum_{i=0}^j x^i P_i\in M_p(\RR[x])$, $0\leq j<n$,  
		be a block column relation of $M(n)$. 
		By \cite[Lemma 5.21]{KT22}, $P$ is a block column relation of $M(\infty)$,
		where the definition of the evaluation $P(X)$ from Subsection \ref{subsec:evaluation} extends from $M(n)$
		to $M(\infty)$ in a natural way.
		By \cite[Lemma 5.15]{KT22}, $(P\cdot xI_p)=xI_p\cdot P=xP$ is a column relation of $M(\infty)$ and in particular, of $M(n)$.
	\end{proof}
	
	\begin{corollary}
		\label{cor:sizes-od-Dep-sets}
		Let $n, p \in \NN$ and let $\mathcal S:=(S_0,S_1,\ldots,S_{2n})\in (\Sym_p(\RR))^{2n+1}$ be a sequence with a representing measure $\mu$.
		For $i\in [0;n]$ 
		let $\Dep(X^i)$ be as in \eqref{eq:defDep-of-Y-i}. 
		Then
		$$0\leq \Card \Dep(X^i)\leq \Card\Dep(X^{i+1})\leq p\quad \text{for } i\in [0;n-1].$$
	\end{corollary}

	\subsection{Change of basis on block columns and rows of the moment matrix}
	
	Let $\cS\equiv\cS^{(2n)}$ be as in \eqref{def:sequence} and let $t \in \RR$.
	Let $\phi(x)=x-t$ be an invertible affine linear transformation
	and
	let 
	\begin{equation*}
		\label{def:sequenceT}
		\cT\equiv\mathcal T^{(2n)}\equiv (T_0,T_1,\ldots,T_{2n})
	\end{equation*}
	be a sequence, defined for each $i\in [0;2n]$ by
	\begin{equation*}
		\label{def:Ti}
		T_i:=
		L_{\cS}((x-t)^i)=
		L_{\cS}\Big(\sum_{\ell=0}^{i}\binom{i}{\ell}(-1)^{\ell}x^{i-\ell}t^{\ell}\Big)=
		\sum_{\ell=0}^{i}\binom{i}{\ell}(-1)^{\ell}S_{i-\ell}t^{\ell}.
	\end{equation*}       
	The following proposition is a version of \cite[Proposition 1.9]{CF05} {for matricial sequences}. 
	
	\begin{proposition}
		\label{linear transform invariance-nc}
		Suppose $\phi$, $\cS$, $\cT$ are as above 
		and let $M_\cS(n)$, $M_{\mathcal{T}}(n)$ be the corresponding moment matrices. We denote by $X_\cS^i$ and $X_\cT^i$ the $i$-th block column of 
		$M_\cS(n)$ and $M_{\mathcal{T}}(n)$, respectively.
		Let $J_\phi: \RR^{n+1}\to \RR^{n+1}$
		be a linear map, given by 
		$$J_\phi \coef(p(x)):=\coef(p(x-t)),$$
		where $\coef$ is as in \eqref{def:coefficients}.
		Let $J_\phi\otimes I_p:\RR^{n+1}\otimes M_p(\RR)\to \RR^{n+1}\otimes M_p(\RR)$, 
		where $\otimes$ stands for the Kronecker product of matrices, i.e., $v \otimes A \mapsto J_\phi v \otimes A$. 
		Then the following statements are true:
		\begin{enumerate}
			\item 
			$(J_\phi \otimes I_p) \coef(P(x))=\coef((P \circ \phi)(x))=\coef(P(x-t)).$
			\item
			\label{linear transform invariance-nc-1}
			$M_{\mathcal{T}}(n)=(J_\phi \otimes I_p)^T
			M_{{\cS}}(n) (J_\phi \otimes I_p).$
			\item 
			\label{linear transform invariance-nc-2}
			$J_\phi$ is invertible and, in particular, $J_\phi \otimes I_p$ is invertible.
			\item  
			\label{linear transform invariance-nc-3}
			$M_{\mathcal{T}}(n)\succeq 0 \Leftrightarrow  M_{{\cS}}(n)\succeq 0.$
			\item 
			\label{linear transform invariance-nc-4}
			$\Rank M_{\mathcal{T}}(n)=\Rank M_{{\cS}}(n).$
			\item 
			\label{linear transform invariance-nc-5}
			For $P(x)=\sum_{i=0}^{n}x^iP_i\in M_p(\RR[x])$, we have 
			$P(X_\cT)=(J_\phi \otimes I_p)^T((P\circ \phi)(X_\cS))$.
		\end{enumerate}
	\end{proposition}
	
	\begin{proof}	
		{Straightforward.}
	\end{proof}
	
	{

		\begin{corollary}
			\label{co:connection-between-both-mm}
			Suppose $n, p \in \NN$, $\cS, \cT, M_{\cS}(n)$
			and $M_{\cT}(n)$ are as above.
			We denote by $X_{\cS}^i$ and $X_{\cT}^i$ the $i$-th block column of 
			$M_\cS(n)$ and 
			$M_{\mathcal{T}}(n)$,
			respectively.
			Assume that $P_{\cT}(x)=\sum_{i=0}^{n}x^iP_i\in M_p(\RR[x])$ is a block column relation of $M_{\cT}(n)$. Then 
			$P_\cS(x):=P_{\cT}(x-t)$
			is a block column relation of $M_{\cS}(n)$.
		\end{corollary}
		
		\begin{proof}
			Let $\phi$ and $J_\phi$ be as in Proposition 
			\ref{linear transform invariance-nc}. 
			By \eqref{linear transform invariance-nc-2} and 
			\eqref{linear transform invariance-nc-5} of
			Proposition \ref{linear transform invariance-nc},
			the evaluation $P_\cS(X_\cS)$ is equal
			to
			$(J_\phi\otimes I_p)^{-T}P_{\cT}(X_\cT)$,
			which implies the statement of the corollary.
		\end{proof}
	}

	\begin{proposition}
		\label{recursive-generation-of-both-mm}
		Let $n, p \in \NN$, $\cS$ and $\mathcal T$ are as in \eqref{def:sequence} and \eqref{def:sequenceT}, and let $M_\cS(n)$ and 
		$M_{\mathcal{T}}(n)$ be the corresponding moment matrices. Then $M_\cS(n)$ is block--recursively generated if and only if $M_{\mathcal{T}}(n)$ is block--recursively generated.
	\end{proposition}
	
	\begin{proof}
		Assume the notation from 
		Proposition \ref{linear transform invariance-nc}.
		By symmetry it suffices to prove that if
		$M_{\mathcal{S}}(n)$ is block--recursively generated,
		it follows that $M_{\mathcal{T}}(n)$ is block--recursively generated.
		Assume that $M_{\mathcal{S}}(n)$ is block--recursively generated.
		Let $P(x)=\sum_{i=0}^j x^i P_i\in M_p(\RR[x])$, $0\leq j<n$,  
		be a block column relation of $M_\cT(n)$.
		By Proposition \ref{linear transform invariance-nc}.\eqref{linear transform invariance-nc-5},
		$(P\circ \phi)(x)$ is a block column relation of $M_\cS(n)$.
		Since $M_\cS(n)$ is block--recursively generated, it follows that 
		$(xP\circ \phi)$ is a block column relation of $M_\cS(n)$.
		By Proposition \ref{linear transform invariance-nc}.\eqref{linear transform invariance-nc-5}, $xP$ is a block column relation of $M_\cT(n)$,
		whence $M_\cT(n)$ is block--recursively generated. 
	\end{proof}

	\subsection{Determinant of a matrix polynomial}
	
	The following lemma, which will be essentially used in the proofs of 
	our main results,
	is a refined version of \cite[Lemma 3.1]{ZZ25}.
	
	\begin{lemma}
		\label{le:determinant-of-matrix-valued-polynomial-v2}
		Let \(p,n \in \mathbb{N}\), \(t \in \mathbb{R}\)
		and 
		$$H(x) = \sum_{i=0}^{n}(x-t)^i H_i\in M_p(\RR[x]_{\leq n})$$
		be a nonzero matrix polynomial.
		For $k\in [0;n]$ we define
		$$
		W_k := \bigcap_{i=0}^{k}\Ker H_i 
		\qquad \text{and} \qquad 
		s_k := \dim W_k.
		$$
		Then
		\begin{equation}
			\label{determinant-H-v2}
			\det H(x) 
			= 
			\left\{
			\begin{array}{rl}
				(x-t)^{\sum_{k=0}^{n-1} s_k} \cdot g(x),&
				\text{if }s_n=0,\\[0.3em]
				0,&
				\text{if }s_n>0,
			\end{array}
			\right.
		\end{equation}
		where 
		$0\neq g(x) \in \mathbb{R}[x]$.
	\end{lemma}
	
	\begin{proof}
		Note that
		$$
		W_{n} \subseteq
		W_{n-1} \subseteq 
		\ldots \subseteq 
		W_0 \subseteq \RR^p 
		\qquad \text{and} \qquad 
		0\leq s_{n} \leq s_{n-1} \leq \ldots \leq s_0 \leq p.
		$$ 
		Clearly, if $s_n>0$, there exists a nonzero vector $v\in \RR^p$ such that $H(x)v=\mathbf 0_{p\times 1}$, which implies 
		\eqref{determinant-H-v2}. From now on we assume that $s_n=0$.
		Let
		$\mathcal{B} := \{b_1, b_2, \ldots, b_p\}$ be a basis of $\RR^p$ 
		such that for each $i\in [0;n-1]$ the set 
		$\{b_1, b_2, \ldots, b_{s_{i}}\}$ is a basis of $W_{i}$.  
		Let us define an invertible matrix
		$$
		B := \begin{pmatrix}
			b_1 & b_2 & \cdots & b_p
		\end{pmatrix} \in M_p(\RR).
		$$
		For $i\in [0;n]$ define matrices
		\begin{equation}
			\label{structure-of-widetildeHi-v2}
			\widetilde{H}_i 
			:= H_iB
			\equiv
			\big(\begin{array}{ccccccc}
				\mathbf{0}_{p\times s_i}
				&
				\widetilde{h}_i^{(s_i+1)} & 
				\cdots &          
				\widetilde{h}_i^{(p)}
			\end{array}\big),
		\end{equation}
		where the first $s_i$ columns of $\widetilde{H}_i$ are zero due to 
		$b_1, b_2, \ldots, b_{s_i}\in \Ker H_i$.
		Define a matrix polynomial
		$$\widetilde{H}(x) := \sum_{i=0}^{n}(x-t)^i\widetilde{H}_i = H(x)B \in M_p(\RR[x]).$$
		By \eqref{structure-of-widetildeHi-v2}, it follows that the first $s_{n-1}$ columns of $\widetilde{H}(x)$
		are of the form
		$$
		(x-t)^n
		\begin{pmatrix}
			\widetilde{h}_{n}^{(1)} &   
			\widetilde{h}_n^{(2)} & \cdots & 
			\widetilde{h}_n^{(s_{n-1})}
		\end{pmatrix},
		$$
		while for every $k\in [n-1]$ the columns $[s_{k}+1;s_{k-1}]$ of
		$\widetilde H(x)$ are equal to
		\begin{align*}
			& 
			(x-t)^k\sum_{i=0}^{n-k}(x-t)^i\begin{pmatrix}
				\widetilde{h}_{i+k}^{(s_k+1)} & \widetilde{h}_{i+k}^{(s_k+2)} & \cdots & \widetilde{h}_{i+k}^{(s_{k-1})}
			\end{pmatrix}.
		\end{align*}
		Observe that the first $s_{n-1}$ columns of $\widetilde{H}(x)$ have a common factor $(x-t)^n$, the next $s_{n-2} - s_{n-1}$ columns a common factor $(x-t)^{n-1}$, etc.
		Using this observation and upon
		factoring the determinant of 
		$\widetilde{H}(x)$
		column-wise, we obtain 
		\begin{align*} 
			\det H(x) 
			= 
			\frac{\det \widetilde{H}(x)}{\det B}
			&=
			(x-t)^{ns_{n-1}}
			(x-t)^{(n-1)(s_{n-2}-s_{n-1})}\cdots
			(x-t)^{s_{0}-s_1}g(x)\\[0.2em]
			&=(x-t)^{ns_{n-1}+\sum_{i=1}^{n-1}i(s_{i-1}-s_i)}g(x)\\
			&=
			(x-t)^{\sum_{k=0}^{n-1}s_k}g(x),
		\end{align*}
		which proves \eqref{determinant-H-v2}. Since $s_n=0$, $g(x) \neq 0$ also holds.
	\end{proof}

	\subsection{Inclusion of kernels}
	The following technical lemma will be used in the proof of the implication $\eqref{th:mainTheorem-pt1} \Rightarrow \eqref{th:mainTheorem-pt2}$ of Theorem \ref{th:mainTheorem}}.
	
	\begin{lemma}
	\label{110722-1428}
	Let $p\in \NN$ and $A,B\in \Sym_p(\RR)$ such that $A\succeq B\succeq 0$. Then $\Ker A\subseteq \Ker B$. 
	\end{lemma}
	
	\begin{proof} 
	Let us take $v\in \Ker A$. 
	From $0=v^TAv\geq v^TBv\geq 0$, it follows that $v^TBv=0$. 
	By
	$0=v^TBv=v^TB^{\frac{1}{2}}B^{\frac{1}{2}} v=\|B^{\frac{1}{2}} v\|^2,$
	it follows that $v\in \Ker B^{\frac{1}{2}}$ and thus $v\in \Ker B.$
	\end{proof}

	\subsection{Characterization of positive semidefiniteness of a $2\times 2$ block matrix}\label{SubS2.1}
	Let 
	\begin{equation*}
	M=\left(  \begin{array}{cc} A & B \\ C & D \end{array}\right)\in M_{n+m}(\RR)
	\end{equation*}
	be a real matrix where 
	$A\in M_n(\RR)$, $B\in M_{n\times m}(\RR)$, $C\in M_{m\times n}(\RR)$  and $D\in M_{m}(\RR)$.
	The \textbf{generalized Schur complement} \cite{Zha05} of $A$ (resp.\ $D$) in $M$ is defined by
	$$M/A=D-CA^\dagger B\quad(\text{resp.}\; M/D=A-BD^\dagger C),$$
	where $A^\dagger $ (resp.\ $D^\dagger $) stands for the Moore-Penrose inverse of $A$ (resp.\ $D$). 
	
	
	The following theorem is a characterization of psd $2\times 2$ block matrices. 
	
	\begin{theorem}[{\cite{Alb69}}]
	\label{block-psd} 
	Let 
	\begin{equation*}
		M=\left( \begin{array}{cc} A & B \\ B^{T} & C\end{array}\right)\in \Sym_{n+m}(\RR)
	\end{equation*} 
	be a real symmetric matrix where $A\in \Sym_n(\RR)$, $B\in M_{n\times m}(\RR)$ and $C\in \Sym_m(\RR)$.
	Then: 
	\begin{enumerate}
		\item 
		\label{021123-1702}
		The following conditions are equivalent:
		\begin{enumerate}
			\item 
			\label{pt1-281021-2128} 
			$M\succeq 0$.
			\smallskip
			\item 
			\label{pt2-281021-2128} 
			$C\succeq 0$,         
			$\cC(B^T)\subseteq\cC(C)$ and $M/C\succeq 0$.
			\smallskip
			\item 
			\label{pt3-281021-2128}
			$A\succeq 0$,    
			$\cC(B)\subseteq\cC(A)$ and $M/A\succeq 0$.
		\end{enumerate}
		\smallskip
		\item 
		\label{prop-2604-1140-eq2}
		If $M\succeq 0$, then 
		$$\Rank M= \Rank A+\Rank M/A=\Rank C+\Rank M/C.$$
	\end{enumerate}
	\end{theorem}

	\subsection{Extension principle}
	
	\begin{lemma}
	[{\cite[Lemma 2.4]{Zal22a}}]
	\label{extension-principle}
	Let $A\in \Sym_n(\RR)$ be positive semidefinite, 
	$Q\subseteq [n]$
	and	
	$A|_Q$ be the restriction of $A$ to the rows and columns from the set $Q$. 
	If $A|_Qv=0$ for a nonzero vector $v$,
	then $A\widehat v=0$ where $\widehat{v}$ is a vector with the only nonzero entries in the rows from $Q$ and such that the restriction 
	$\widehat{v}|_Q$ to the rows from $Q$ equals to $v$. 
	\end{lemma}

	\subsection{Lemma on the rank of the matrix}
	Another technical lemma that will be used in the proof of the implication $\eqref{th:mainTheorem-pt2} \Rightarrow \eqref{th:mainTheorem-pt1}$ of Theorem \ref{th:mainTheorem} is the following.

	\begin{lemma}
	\label{lemma-rank-equalities}
	Let $m_1, m_2, c, k, \ell \in \NN$ and $A_1 \in M_{m_1\times k}(\RR)$, $A_2 \in M_{m_2\times k}(\RR)$, $C \in M_{c\times k}(\RR)$, $B_1 \in M_{m_1\times \ell}(\RR)$, $B_2 \in M_{m_2\times \ell}(\RR)$, $D \in M_{c\times \ell}(\RR)$. Assume the following conditions hold:
	\begin{enumerate}
		\item 
		\label{lemma-rank-equalities-condition1}
		$\Rank A_2= 
		\Rank \begin{pmatrix}
			A_1 \\
			A_2
		\end{pmatrix} = \Rank \begin{pmatrix}
			A_1 & B_1 \\
			A_2 & B_2
		\end{pmatrix}$,
		\item
		\label{lemma-rank-equalities-condition2}
		$\Rank \begin{pmatrix}
			A_2 \\
			C
		\end{pmatrix} = \Rank \begin{pmatrix}
			A_2 & B_2 \\
			C & D
		\end{pmatrix}$,
	\end{enumerate}
	Then $$\Rank \begin{pmatrix}
		A_1 \\
		A_2 \\
		C
	\end{pmatrix} = \Rank \begin{pmatrix}
		A_1 & B_1 \\
		A_2 & B_2 \\
		C & D
	\end{pmatrix}.$$
	\end{lemma}
	
	\begin{proof}
	The assumptions 
	\eqref{lemma-rank-equalities-condition1} and \eqref{lemma-rank-equalities-condition2}
	imply that there exist
	matrices
	$U \in M_{m_1\times m_2}(\RR)$,
	$V \in M_{k\times \ell}(\RR)$,
	$W \in M_{k\times \ell}(\RR)$,
	such that
	\begin{equation}
		\label{matricial-equalities}
		A_1 = UA_2,\quad
		B_1 = A_1V,\quad 
		B_2 = A_2V,\quad
		B_2 = A_2W,\quad 
		D = CW.
	\end{equation}
	Using \eqref{matricial-equalities} we get
	$$B_1 = A_1V = UA_2V = UB_2 = UA_2W = A_1W$$ and therefore
	$$\begin{pmatrix}
		B_1 \\
		B_2 \\
		D
	\end{pmatrix} = \begin{pmatrix}
		A_1W \\
		A_2W \\
		CW
	\end{pmatrix} = \begin{pmatrix}
		A_1 \\
		A_2 \\
		C
	\end{pmatrix}W,$$ which completes the proof.
	\end{proof}

	\section{Proof of the implication $\eqref{th:mainTheorem-pt2} \Rightarrow \eqref{th:mainTheorem-pt1}$ and the moreover part of Theorem \ref{th:mainTheorem}}
	\label{sec:proofOfMainTheorem}
	
	Let $S_i:=L(x^i)$, $T_i:=L((x-t)^i)$ for $i\in [0;2n]$
	and $\cS:=(S_i)_{i\in [0;2n]}$, $\cT:=(T_i)_{i\in [0;2n]}$.
	Let $X_\cT^i \equiv 
	\row\begin{pmatrix}
		(x_\cT^{(i)})_j     
	\end{pmatrix}_{j\in [p]}$,
	$\Dep(X_\cT^i)$, $\Depn(X_\cT^i)$
	and $\Depn_{\ell}(X_\cT^i)$, $\ell=0,1$,
	be
	as in \eqref{columns-of-Yi},
	\eqref{eq:defDep-of-Y-i} and \eqref{eq:defDepnew-of-Y-i}, 
	respectively. Write
	\begin{equation*}
		(x_\cT^{(i)})_j  =: \begin{pmatrix}
			u_j^{(i)} \\
			v_j^{(i)}
		\end{pmatrix} \in \begin{pmatrix}
			\RR^{np}\\
			\RR^p 
		\end{pmatrix}
		\qquad
		\text{and}
		\qquad
		\col(T_{n+j})_{j\in [n]}
		=:\row(k_i)_{i\in [p]}\in \RR^{np}.
	\end{equation*}
	Note that in the formulation of Theorem \ref{th:mainTheorem}, we have
	\begin{align*} 
		\mathbf{y}_{ip+j} &= (x_\cT^{(i)})_j,
		\qquad
		\mathbf{u}_{ip+j} =
		\left\{
		\begin{array}{rl}
			u_j^{(i)},&    
			\text{for }i\in [0;n],\\[0.5em] 
			k_j,&    
			\text{for }i=n+1,
		\end{array}
		\right.\\
		A &= \bigcup_{i=1}^n 
		\Depn_1(X^i_\cT),\\
		B &= \Big \{j \in 	[p]\colon j \notin \Dep(X_\cT^n),\; 
		k_j\in \Span\{
		u^{(1)}_1,u^{(1)}_2,\ldots,
		u^{(n)}_p,
		k_1,\ldots,k_{j-1}\} \Big \}. 
	\end{align*}

	In the first step of the proof 
	we will build relation matrices that represent column relations of $M_{\cT}(n)$.
	By Proposition \ref{recursive-generation-if-measure}, $M_\cS(n)$ is block--recursively generated and by Proposition \ref{recursive-generation-of-both-mm}, so is $M_{\mathcal T}(n)$.
	Let 
	\begin{equation}
		\label{eq:defmathcalZ-v2}
		\mathcal{Z} \equiv \{
		z_1,\ldots,z_s\colon 0\leq z_1<z_2<\ldots<z_s\leq n
		\}
		:= \big\{i \in [0;n] 
		\colon
		\Depn(X_\cT^i) \neq \emptyset \big\}
		.
	\end{equation}
	Let $Q \in M_p(\RR)$ be a permutation matrix which permutes the columns of the given matrix, when multiplied from the right, to the order
	\begin{align}
		\label{order-permutation-Q}
		\begin{split}
			&[p] \setminus \Dep(X_{\cT}^{z_s}),\;           
			\Depn_0(X_{\cT}^{z_s}),\;
			\Depn_1(X_{\cT}^{z_s}),\; 
			\Depn_0(X_{\cT}^{z_{s-1}}),\; 
			\Depn_1(X_{\cT}^{z_{s-1}}),\;
			\ldots,\\
			&\Depn_0(X_{\cT}^{z_2}),\; 
			\Depn_1(X_{\cT}^{z_2}),\;
			\Depn_0(X_{\cT}^{z_1}),\;
			\Depn_1(X_{\cT}^{z_1}).
		\end{split}
	\end{align}
	Let
	\begin{align}
		\label{def-of-p0}
		\begin{split}
			p_0 &:= p-\Card\Dep(X_{\cT}^{z_s}),\\
			p_{j,\ell} &:=\Card\Depn_\ell(X_\cT^{z_{s
					-j+1}})\;\;\text{for }j\in [s] \text{ and }\ell=0,1,\\
			p_j &:=\Card\Depn(X_\cT^{z_{s-j+1}})=p_{j,0}+p_{j,1}
			\;\;\text{for }j\in [s].
		\end{split}
	\end{align}
	Note that $\Card A=\sum_{j=1}^s p_{j,1}$.\\
	
	\noindent \textbf{Claim 1.}
	We have 
	\begin{align}
		\label{031025-eq1}
		\Rank M(n)
		&=\Rank M_{\cT}(n)=
		(n+1)p_0 + \sum_{j=1}^{s}z_{s-j+1}p_j,\\
		\label{031025-eq2}
		\Rank \row(X^{i}_{\cT})_{i\in [n]}
		&=\Rank M_{\cT}(n)-p_0-\Card A.
	\end{align}
	
	\noindent \textit{Proof of Claim 1.}
	\eqref{031025-eq1} follows by definition of $p_j$ and $z_j$, together with
	the fact that
	$M_{\cT}(n)$ is block--recursively generated, which in particular implies that 
	if $i\in \Dep(X_{\cT}^{j})$ for some $j\in [0;n-1]$, then $i\in \Dep(X_{\cT}^{j+1}).$
	For \eqref{031025-eq2} note by the same reasoning that
	\begin{align*} 
		\Rank \row(X^{i}_{\cT})_{i\in [n]}
		&=  np_0 + \sum_{j=1}^{s}z_{s-j+1}p_{j,0}
		+\sum_{j=1}^{s}(z_{s-j+1}-1)p_{j,1}\\
		&= \Rank M_{\cT}(n)-p_0-\sum_{j=1}^s p_{j,1}\\
		&=\Rank M_{\cT}(n)-p_0-\Card A.
	\end{align*}
	In the first line we used the fact that if 
	$j\in \Depn_0(X_\cT^{z_{s-j+1}})$, then
	each column
	$(x_\cT^{(i)})_j$ is not in the span of the previous columns
	for $i\in[0;z_{s-j+1}-1]$, but $(x_\cT^{(z_{s-j+1})})_j$ is in the span of the previous columns, while it is not in the span of the columns
	$(x_\cT^{(1)})_1, (x_\cT^{(1)})_2, \ldots, (x_\cT^{(z_{s-j+1})})_{j-1}$.
	However, by recursive generation, $j\in \Dep_1(X_{\cT}^{z_{s-j+1}+1})$. 
	\hfill$\blacksquare$\\
	
	Let us define the permuted matrix
	\begin{align}
		\begin{split}
			\label{def:permuted}
			\row(\widecheck{X}_\cT^i)_{i\in [0;n]}
			:=
			\row(X_\cT^iQ)_{i\in [0;n]}.
		\end{split}
	\end{align} 
	Write
	\begin{equation}
		\label{block-decomposition-of-widehatYzj-semidefinite-case}
		\widecheck{X}_\cT^{i} =: 
		\big(\underbrace{(\widecheck{X}_\cT^i)_{(0)}}_{\substack{p_0\\\text{columns}}} \; 
		\underbrace{(\widecheck{X}_\cT^i)_{(1,0)}}_{\substack{p_{1,0}\\\text{columns}}} \; 
		\underbrace{(\widecheck{X}_\cT^i)_{(1,1)}}_{\substack{p_{1,1}\\\text{columns}}} \;
		\underbrace{(\widecheck{X}_\cT^{i})_{(2,0)}}_{\substack{p_{2,0}\\\text{columns}}} \; 
		\underbrace{(\widecheck{X}_\cT^{i})_{(2,1)}}_{\substack{p_{2,1}\\\text{columns}}} \;
		\cdots \; 
		\underbrace{(\widecheck{X}_\cT^i)_{(s,0)}}_{\substack{p_{s,0}\\\text{columns}}}\;
		\underbrace{(\widecheck{X}_\cT^i)_{(s,1)}}_{\substack{p_{s,1}\\\text{columns}}}\;
		\big).
	\end{equation}
	\smallskip

	\noindent \textbf{Claim 2.}
	There exists matrices 
	$H_0,H_1,\ldots,H_{n-1} \in M_{p\times (p-p_0)}(\RR),\widecheck{H}_{n,(0)}\in M_{p_0\times (p-p_0)}(\RR)$ 
	and an invertible matrix $\widecheck{H}_n\in M_{p-p_0}(\RR)$
	such that for
	\begin{equation}
		\label{def:G}
		G:=\col\big(
		\col(H_i)_{i\in [0;n-1]}, \widecheck{H}_{n,(0)}
		\big),
	\end{equation}
	we have
	\begin{align}
		\begin{split}
			\label{column-relation-v6-semidefinite-case}
			\row\big(
			(\widecheck{X}_{\cT}^{n})_{(i,0)}\;
			(\widecheck{X}_{\cT}^{n})_{(i,1)}
			\big)_{i\in [s]}\widecheck{H}_n
			&= 
			\begin{pmatrix} 
				\row({X}_{\cT}^{i})_{i\in [0;n-1]} & (\widecheck{X}_{\cT}^{n})_{(0)}
			\end{pmatrix}
			G.
		\end{split}
	\end{align}
	For $i\in [0;n-1]$ let $W_i:=\bigcap_{j=0}^i\Ker H_j$ and $s_i:=\dim W_i$. Write $r_j := \sum_{\ell=1}^{j}p_{\ell}$ for $j\in [s]$. Then
	\begin{align}
		\label{sum-of-dim-ker-Hi-proposition-psd-case}
		\begin{split}
			&\sum_{i=0}^{n-1}s_i
			\geq
			(n-z_s)(p-p_0) 
			+ \sum_{j=1}^{s-1}(z_{j+1}-z_j)(p-p_0-r_{s-j})
			+\Card A.
		\end{split}
	\end{align}
	
	\noindent \textit{Proof of Claim 2.}
	Let $r_0 := 0$. 
	By definition of the indices $z_j$,
	for $j\in [s]$ and $\ell=0,1$ there exist matrices 
	\begin{align}
		\label{relation-matrices}
		\begin{split}
			&\widecheck{B}_0^{(j,\ell)}, \widecheck{B}_1^{(j,\ell)}, \ldots, \widecheck{B}_{z_j - 1}^{(j,\ell)} \in M_{p\times p_{s-j+1,\ell}}(\RR), \quad
			\widecheck{B}_{z_j,(0)}^{(j,\ell)} \in M_{p_0\times p_{s-j+1,\ell}}(\RR)
			\quad\text{and}\quad\\ 
			&\widecheck{B}_{z_j}^{(j,\ell)} \in M_{(r_{s-j}+\sum_{i=0}^{\ell-1}p_{s-j+1,i})\times p_{s-j+1,\ell}}(\RR)
		\end{split}
	\end{align} 
	(note that if $\ell = 0$, then $\sum_{i=0}^{\ell-1}p_{s-j+1,i} = 0$) such that 
	\begin{align}
		\label{column-relation-semidefinite-case}
		\begin{split}& (\widecheck{X}_{\cT}^{z_j})_{(s-j+1,\ell)} = 
			\row(\widecheck{X}_{\cT}^{i})_{i\in [0;z_j-1]}
			\col(\widecheck{B}_{i}^{(j,\ell)})_{i\in [0;z_j-1]}+(\widecheck{X}_{\cT}^{z_j})_{(0)}\widecheck{B}_{z_j,(0)}^{(j,\ell)}\\
			&\hspace{0.8em}+
			\begin{pmatrix}
				\row\big(
				(\widecheck{X}_{\cT}^{z_j})_{(i,0)} \; 
				(\widecheck{X}_{\cT}^{z_j})_{(i,1)}
				\big)_{i\in [s-j]} &\row\big((\widecheck{X}_{\cT}^{z_j})_{(s-j+1,i)}\big)_{i \in [0;\ell-1]}
			\end{pmatrix}
			\widecheck{B}_{z_j}^{(j,\ell)}.
		\end{split}
	\end{align}
	Note that if $\ell = 0$, then $\row\big((\widecheck{X}_{\cT}^{z_j})_{(s-j+1,i)}\big)_{i \in [0;\ell-1]}$ is an empty matrix. By the ordering of the columns (see \eqref{order-permutation-Q}), note that $\widecheck{B}_0^{(j,1)} = \mathbf{0}_{p\times p_{s-j+1,\ell}}$ for each $j \in [s]$.
	By Proposition \ref{recursive-generation-if-measure}, 
	multiplying the relations in
	\eqref{column-relation-semidefinite-case} by $\widecheck{X}_\cT^{n-z_j}$, we obtain
	for $j\in [s]$ and $\ell=0,1$
	the relations
	\begin{align*}
		\begin{split}
			&
			\begin{pmatrix} 
				\row\big(
				(\widecheck{X}_{\cT}^n)_{(i,0)} \; 
				(\widecheck{X}_{\cT}^n)_{(i,1)}
				\big)_{i\in [s-j]} &\row\big((\widecheck{X}_{\cT}^n)_{(s-j+1,i)}\big)_{i \in [0;\ell-1]}
				&
				(\widecheck{X}_{\cT}^{n})_{(s-j+1,\ell)} 
			\end{pmatrix}
			\begin{pmatrix} 
				-\widecheck{B}_{z_j}^{(j,\ell)}\\[0.2em]
				I_{p_{s-j+1,\ell}}
			\end{pmatrix}\\[0.5em]
			=& 
			\row(\widecheck{X}_{\cT}^{n-z_j+i})_{i\in [0;z_j-1]}
			\col(\widecheck{B}_{i}^{(j,\ell)})_{i\in [0;z_j-1]}
			+
			(\widecheck{X}_{\cT}^{n})_{(0)}
			\widecheck{B}_{z_j,(0)}^{(j,\ell)}
		\end{split}
	\end{align*}
	of $M_\cT(n)$,
	or equivalently for $j\in [s]$ we have
	\begin{align}
		\label{column-relation-v3-semidefinite-case}
		\begin{split}
			\row\big(
			(\widecheck{X}_{\cT}^{n})_{(i,0)}\;
			(\widecheck{X}_{\cT}^{n})_{(i,1)}\;
			\big)_{i\in [s]}
			\widecheck{H}_{n}^{(j)} &= 
			\sum_{i=0}^{n-1} \widecheck{X}_\cT^i\widecheck{H}_{i}^{(j)}
			+
			(\widecheck{X}_{\cT}^{n})_{(0)}\widecheck{H}_{n,(0)}^{(j)},
		\end{split}
	\end{align}
	where
	\begin{align*}
		\begin{split}
			\widecheck{H}_i^{(j)} &:=
			\mathbf{0}_{p\times (p-p_0)}\quad \text{for }\; i\in [0;n-z_j-1],\\[0.2em]
			\widecheck{H}_i^{(j)} 
			&:= \big(
			\mathbf{0}_{p\times r_{s-j}}\;
			\widecheck{B}_{i-n+z_j}^{(j,0)}\;
			\widecheck{B}_{i-n+z_j}^{(j,1)}\;
			\mathbf{0}_{p\times (p-p_0-r_{s-j+1})}
			\big)\in M_{p\times (p-p_0)}(\RR)
			\quad \text{for }\; i\in [n-z_j;n-1],\\[0.3em]
			\widecheck{H}_{n,(0)}^{(j)} 
			&:= \big(
			\mathbf{0}_{p_0\times r_{s-j}}\;
			\widecheck{B}_{z_j,(0)}^{(j,0)}\;
			\widecheck{B}_{z_j,(0)}^{(j,1)}\;
			\mathbf{0}_{p_0\times (p-p_0-r_{s-j+1})}
			\big)\in M_{p_0\times (p-p_0)}(\RR),\\[0.3em]
			\widecheck{H}_{n}^{(j)} 
			&:= 
			\begin{pmatrix}
				\mathbf 0_{r_{s-j}} & 
				\begin{pmatrix} -\widecheck{B}_{z_j}^{(j,0)} 
					\\
					I_{p_{s-j+1,0}}
				\end{pmatrix}
				&\widecheck{B}_{z_j}^{(j,1)}\\[0.3em]
				\mathbf{0} & \mathbf{0}_{p_{s-j+1,1}\times p_{s-j+1,0}} & 
				I_{p_{s-j+1,1}}
			\end{pmatrix} \oplus \mathbf{0}_{p-p_0-r_{s-j+1}}
			\in M_{p-p_0}(\RR).
		\end{split}
	\end{align*}
	Let
	\begin{align*}
		\widecheck{B}_{z_j}^{(j,0)} &=: \col(\widecheck{B}_{z_j,i}^{(j,0)})_{i\in [s-j]} \quad \text{and} \\
		\widecheck{B}_{z_j}^{(j,1)} &=: \col\big(\col
		(\widecheck{B}_{z_j,i}^{(j,1)})_{i\in [s-j]}, \widecheck{B}_{z_j,s-j+1,0}^{(j,1)} \big)
		\quad \text{for }j\in [s],
	\end{align*}
	where $\widecheck{B}_{z_j,i}^{(j,\ell)} \in M_{p_i\times p_{s-j+1,\ell}}(\RR)$, $\widecheck{B}_{z_j,s-j+1,0}^{(j,1)} \in M_{p_{s-j+1,0}\times p_{s-j+1,1}}(\RR)$ for $j\in [s]$, $i \in [s-j]$ and $\ell = 0,1$.
	Summing all block column relations  \eqref{column-relation-v3-semidefinite-case} for $j\in [s]$, we get a block column relation
	\begin{equation}
		\label{column-relation-v4-semidefinite-case}
		\row\big(
		(\widecheck{X}_{\cT}^{n})_{(i,0)}\;
		(\widecheck{X}_{\cT}^{n})_{(i,1)}
		\big)_{i\in [s]}
		\widecheck{H}_{n}
		= 
		\sum_{i=0}^{n-1}\widecheck{X}_\cT^i\widecheck{H}_i + (\widecheck{X}_{\cT}^{n})_{(0)}\widecheck{H}_{n,(0)},
	\end{equation}
	where 
	\begin{align}
		\label{definition-check-H-i-psd-case}
		\begin{split}
			&\widecheck{H}_i 
			:= \sum_{j=1}^{s}\widecheck{H}_i^{(j)}\\
			=&
			\left\{
			\begin{array}{l}
				\mathbf{0}_{p\times (p-p_0)},    
				\qquad \text{for }i\in [0;n-z_s-1],\\[0.7em] 
				\Big(
				\underbrace{\widecheck{B}_{i-n+z_{s}}^{(s,0)}}_{\substack{p_{1,0}\\\text{columns}}} 
				\;
				\underbrace{\mathbf{0}}_{\substack{p-p_0-p_{1,0}\\\text{columns}}}
				\Big),\qquad
				\text{for }i = n-z_{s},\\[2em]
				\Big(
				\underbrace{\widecheck{B}_{i-n+z_{s}}^{(s,0)}}_{\substack{p_{1,0}\\\text{columns}}} 
				\;
				\underbrace{\widecheck{B}_{i-n+z_{s}}^{(s,1)}}_{\substack{p_{1,1}\\\text{columns}}} 
				\;
				\underbrace{\mathbf{0}}_{\substack{p-p_0-p_1\\\text{columns}}}
				\Big),\qquad
				\text{for }i\in [n-z_{s}+1;n-z_{s-1}-1],\\[2em]
				\Big(
				\underbrace{\widecheck{B}_{i-n+z_{s}}^{(s,0)}}_{\substack{p_{1,0}\\\text{columns}}} 
				\;
				\underbrace{\widecheck{B}_{i-n+z_{s}}^{(s,1)}}_{\substack{p_{1,1}\\\text{columns}}} 
				\;
				\underbrace{\widecheck{B}_{i-n+z_{s-1}}^{(s-1,0)}}_{\substack{p_{2,0}\\\text{columns}}} 
				\;
				\underbrace{\mathbf{0}}_{\substack{p-p_0-p_1-p_{2,0}\\\text{columns}}}
				\Big),\;
				\qquad \text{for }i = n-z_{s-1},\\[0.5em]
				\Big(
				\underbrace{\widecheck{B}_{i-n+z_{s}}^{(s,0)}}_{\substack{p_{1,0}\\\text{columns}}} 
				\;
				\underbrace{\widecheck{B}_{i-n+z_{s}}^{(s,1)}}_{\substack{p_{1,1}\\\text{columns}}} 
				\;
				\underbrace{\widecheck{B}_{i-n+z_{s-1}}^{(s-1,0)}}_{\substack{p_{2,0}\\\text{columns}}} 
				\;
				\underbrace{\widecheck{B}_{i-n+z_{s-1}}^{(s-1,1)}}_{\substack{p_{2,1}\\\text{columns}}} 
				\;
				\underbrace{\mathbf{0}}_{\substack{p-p_0-p_1-p_2\\\text{columns}}}
				\Big),\\[3em]
				\hfill\text{for }i\in [n-z_{s-1}+1; n-z_{s-2}-1],\\[0.5em]
				\vdots\\[0.7em]
				\Big(\row\big(\underbrace{\widecheck{B}_{i-n+z_{s-j+1}}^{(s-j+1,0)}}_{\substack{p_{j,0}\\\text{columns}}} 
				\;
				\underbrace{\widecheck{B}_{i-n+z_{s-j+1}}^{(s-j+1,1)}}_{\substack{p_{j,1}\\\text{columns}}}\big)_{j\in [s-1]}\;
				\underbrace{\widecheck{B}_{i-n+z_{1}}^{(1,0)}}_{\substack{p_{s,0}\\\text{columns}}} 
				\;
				\underbrace{\mathbf{0}}_{\substack{p-p_0-r_{s-1}-p_{s,0}\\\text{columns}}}
				\Big),\\[3em]
				\hfill\text{for }i = n-z_{1},\\[0.5em]
				\row\big(\underbrace{\widecheck{B}_{i-n+z_{s-j+1}}^{(s-j+1,0)}}_{\substack{p_{j,0}\\\text{columns}}} 
				\;
				\underbrace{\widecheck{B}_{i-n+z_{s-j+1}}^{(s-j+1,1)}}_{\substack{p_{j,1}\\\text{columns}}}\big)_{j\in [s]},
				\hfill\text{for }i\in [n-z_{1}+1; n-1],
			\end{array}
			\right.
		\end{split}
	\end{align}
	and
	\begin{align*}
		\widecheck H_{n,(0)}
		&=
		\row\big(\underbrace{\widecheck{B}_{z_{s-j+1,(0)}}^{(s-j+1,0)}}_{\substack{p_{j,0}\\\text{columns}}} 
		\;
		\underbrace{\widecheck{B}_{z_{s-j+1,(0)}}^{(s-j+1,1)}}_{\substack{p_{j,1}\\\text{columns}}}\big)_{j\in [s]},
		\\[0.5em]
		\widecheck{H}_{n}
		&=
		\begin{pmatrix}
			D_1 & 
			-\row\big(
			\widecheck{B}_{z_{s-1},1}^{(s-1,\ell)}
			\big)_{\ell=0,1}
			& 
			-\row\big(
			\widecheck{B}_{z_{s-2},1}^{(s-2,\ell)}
			\big)_{\ell=0,1}
			& \cdots & 
			-\row\big(
			\widecheck{B}_{z_{1},1}^{(1,\ell)}
			\big)_{\ell=0,1}\\[0.5em]
			\mathbf{0} & D_2 & 
			-\row\big(
			\widecheck{B}_{z_{s-2},2}^{(s-2,\ell)}
			\big)_{\ell=0,1}
			& \cdots & 
			-\row\big(
			\widecheck{B}_{z_{1},1}^{(1,\ell)}
			\big)_{\ell=0,1}\\[0.5em]
			\vdots &      \ddots   & D_3 & \ddots & \vdots \\[0.5em]
			\vdots &         &    \ddots     & \ddots & 
			-\row\big(
			\widecheck{B}_{z_{1},s-1}^{(1,\ell)}
			\big)_{\ell=0,1}\\[0.5em]
			\mathbf{0} & \cdots      &  \cdots         &    \mathbf{0}   &D_s
		\end{pmatrix},
	\end{align*}
	where $D_j := \begin{pmatrix}
		I_{p_{j,0}} & -\widecheck{B}_{z_{s-j+1},j,0}^{(s-j+1,1)}\\
		\mathbf{0} & I_{p_{j,1}}
	\end{pmatrix}$ for $j \in [s]$. Note that $\widecheck{H}_n$ is invertible. It follows from \eqref{column-relation-v4-semidefinite-case} that
	\begin{align*}
		\begin{split}
			\row\big(
			(\widecheck{X}_{\cT}^{n})_{(i,0)}\;
			(\widecheck{X}_{\cT}^{n})_{(i,1)}
			\big)_{i\in [s]}\widecheck{H}_n
			&= \sum_{i=0}^{n-1}{X}_{\cT}^{i} H_i+(\widecheck{X}_{\cT}^{n})_{(0)}\widecheck{H}_{n,(0)}\\
			&=
			\begin{pmatrix} 
				\row({X}_{\cT}^{i})_{i\in [0;n-1]} & (\widecheck{X}_{\cT}^{n})_{(0)}
			\end{pmatrix}
			G.
		\end{split}
	\end{align*}
	where $H_i := Q\widecheck{H}_i$ for $i\in [0;n-1]$
	and 
	$G$ is as in \eqref{def:G}.
	This proves \eqref{column-relation-v6-semidefinite-case}.
	It remains to prove \eqref{sum-of-dim-ker-Hi-proposition-psd-case}.
	For $i\in [0;n-1]$ let $s_i$ be defined as in the statement of Claim 2.
	Using \eqref{definition-check-H-i-psd-case} we have
	\begin{align}
		\label{kernel-dimensions-psd-case}
		\begin{split}
			s_i
			&=
			\left\{
			\begin{array}{rl}
				p-p_0,& \text{for }i\in [0;n-z_s-1],\\[0.2em]
				\text{at least } p-p_0-p_{1,0},&
				\text{for }i = n-z_s,\\[0.2em]
				\text{at least } p-p_0-p_1,&
				\text{for }i\in [n-z_s+1;n-z_{s-1}-1],\\[0.2em]
				\text{at least } p-p_0-p_1-p_{2,0},&
				\text{for }i = n-z_{s-1},\\[0.2em]
				\text{at least } p-p_0-p_1-p_2,&
				\text{for }i\in [n-z_{s-1}+1;n-z_{s-2}-1],\\[0.2em]
				\vdots&\\
				\text{at least } p-p_0-r_{s-2}-p_{s-1,0},&
				\text{for }i = n-z_{2},\\[0.2em]
				\text{at least } p-p_0-r_{s-1},&
				\text{for }i\in [n-z_{2}+1;n-z_{1}-1],\\[0.2em]
				\text{at least } p-p_0-r_{s-1}-p_{s,0},&
				\text{for }i = n-z_{1}.
			\end{array}
			\right.
		\end{split}
	\end{align}
	It follows from \eqref{kernel-dimensions-psd-case} that
	\begin{align*}
		\begin{split}
			\sum_{i=0}^{n-1}s_i
			&\geq (n-z_s)(p-p_0) 
			+ \sum_{j=1}^{s}(p-p_0-r_{s-j}-p_{s-j+1,0})\\
			&\hspace{11em}+
			\sum_{j=1}^{s-1}(z_{j+1}-z_j-1)(p-p_0-r_{s-j})\\
			&= (n-z_s)(p-p_0) + \sum_{j=1}^{s-1}(z_{j+1}-z_j)(p-p_0-r_{s-j}) + \underbrace{p-p_0-\sum_{j=1}^{s}p_{s-j+1,0}}_{=\sum_{j=1}^{s}p_{j,1}=\Card A},
		\end{split}
	\end{align*}
	which is \eqref{sum-of-dim-ker-Hi-proposition-psd-case}.
	This concludes the proof of Claim 2.\hfill$\blacksquare$\\
	
	The continuation of the proof will be similar to the proof of implication $(2) \Rightarrow (1)$ in \cite[Theorem 1.1]{ZZ25}, while accounting for the column relations of the matrix $M_{\mathcal T}(n)$, which are transferred to the matrix $K$ in \eqref{def-mathcalH0-etc} by recursive generation.
	Write 
	\begin{align}
		\label{def-mathcalH0-etc}
		X_{\cT}^0 =: \begin{pmatrix}
			\mathcal{H}_0 \\
			T_n
		\end{pmatrix} \quad \text{and} \quad 
		\row(X_{\cT}^i)_{i\in [n]}=:
		\begin{pmatrix}
			\mathcal{H} \\
			K^T
		\end{pmatrix}\in 
		\begin{pmatrix}
			\Sym_{np}(\RR)\\ M_{p\times np}(\RR)
		\end{pmatrix}.
	\end{align}
	
	Let 
	\begin{equation} 
		\label{rank-increasement-semidefinite-case}
		k := 
		\Rank 
		\begin{pmatrix}
			\mathcal{H} \\ K^T
		\end{pmatrix} 
		- 
		\Rank \mathcal{H}
		=
		\Rank 
		\begin{pmatrix}
			\mathcal{H} & K
		\end{pmatrix} 
		- 
		\Rank \mathcal{H}.
	\end{equation}
	Let $B$ be as in the statement of Theorem \ref{th:mainTheorem}.
	Note that 
	$$
	B\subseteq [p] \setminus \Dep(X_{\cT}^{z_s})
	\qquad\text{and}\qquad
	\Card\big([p] \setminus 
	\big(B\sqcup \Dep(X_{\cT}^{z_s})\big)\big)=k.
	$$
	(Here $\sqcup$ stands for the disjoint union.)
	Let $P_1\in M_{p_0}(\RR)$ be a permutation matrix,
	which permutes the columns in 
	$[p] \setminus \Dep(X_{\cT}^{z_s})$
	to the order 
	$C_1,C_2$,
	where
	$$
	[p] \setminus 
	\big(B\sqcup \Dep(X_{\cT}^{z_s})\big)
	\subseteq 
	C_1
	\qquad\text{and}\qquad
	k\leq \Card C_1:=p_0-m+\Card A.
	$$
	By the assumption $\Card A \leq m \leq \Card(A \cup B)$, it follows that $0 \leq \Card C_1 \leq p_0$. Recall that $Q$ permutes the columns of the given matrix, when multiplied from the right, to the order
	\eqref{order-permutation-Q}.
	Then for the permutation matrix
	$P := Q(P_1 \oplus I_{p-p_0})\in M_p(\RR)$ and
	\begin{align}
		\label{def-widehatK}
		\widehat{K} \equiv 
		\big(
		\underbrace{\widehat{K}_1}_{\substack{p_0-m+\Card A\\\text{columns}}} \;
		\underbrace{\widehat{K}_2}_{\substack{m-\Card A\\\text{columns}}}  \;
		\underbrace{\widehat{K}_3}_{\substack{p-p_0\\\text{columns}}} 
		\big) 
		:= KP,
	\end{align}
	we have that 
	\begin{equation}
		\label{rank-equalities-semidefinite-case}
		\Rank 
		\begin{pmatrix}
			\mathcal{H} & \widehat{K}_1
		\end{pmatrix} 
		= 
		\Rank 
		\begin{pmatrix}
			\mathcal{H} & \widehat{K}
		\end{pmatrix} 
		= 
		\Rank \mathcal{H} + k.	
	\end{equation}
	By
	\eqref{column-relation-v6-semidefinite-case} and recursive generation (see Proposition \ref{recursive-generation-if-measure}),
	it follows that  
	\begin{equation}
		\label{hat-K-3-psd-case}
		\widehat{K}_3 = \begin{pmatrix}
			\mathcal{H} & \widehat{K}_1 & \widehat{K}_2
		\end{pmatrix} \widehat{G},
	\end{equation}
	where
	\begin{equation}
		\label{def-widehatG}
		\widehat{G} := (I_{np}\oplus P_1^T) G \widecheck{H}_n^{-1},
	\end{equation}
	or equivalently
	\begin{equation}
		\label{rank-equality-v11-semidefinite-case}
		\begin{pmatrix}
			\mathcal{H} & 
			\widehat{K}_1 & 
			\widehat{K}_2 &
			\widehat{K}_3
		\end{pmatrix}  
		\begin{pmatrix}
			-\widehat{G} \\
			I_{p-p_0}
		\end{pmatrix}
		= \mathbf{0}_{np\times (p-p_0)}.
	\end{equation}
	By \eqref{rank-equalities-semidefinite-case}, it     
	follows that 
	\begin{equation}
		\label{definition-hat-K-2-psd}
		\widehat{K}_2 = \begin{pmatrix}
			\mathcal{H} & \widehat{K}_1
		\end{pmatrix} J
	\end{equation}
	for some real matrix  
	$J\in M_{(np + p_0-m+\Card A) \times (m-\Card A)}(\RR)$,
	or equivalently
	\begin{equation}
		\label{rank-equality-v2-semidefinite-case}
		\begin{pmatrix}
			\mathcal{H} & 
			\widehat{K}_1 & 
			\widehat{K}_2
		\end{pmatrix}  
		\begin{pmatrix}
			-J \\
			I_{m-\Card A}
		\end{pmatrix}
		=\mathbf{0}_{np\times (m-\Card A)}
	\end{equation}
	\smallskip
	
	\noindent \textbf{Claim 3.} There exists 
	$\widehat{Z} \in \Sym_p(\RR)$ such that 
	\begin{equation}
		\label{choice-of-hat-Z-semidefinite-case}
		\Rank
		\begin{pmatrix}
			\cH & \widehat{K}\\
			\widehat{K}^T & \widehat{Z}
		\end{pmatrix}
		=
		\Rank 
		\begin{pmatrix} 
			\mathcal{H}\\
			\widehat{K}^T
		\end{pmatrix}    
		+ p_0-m+\Card A
	\end{equation}
	and
	\begin{equation}
		\label{eq:to-prove-280925}
		\begin{pmatrix}
			\mathcal{H} & \widehat{K} \\
			\widehat{K}^T & \widehat{Z}
		\end{pmatrix}
		\begin{pmatrix}
			-J_1 & -\widehat{G}_1\\
			-J_2 & -\widehat{G}_2 \\
			I_{m-\Card A} & -\widehat{G}_3 \\
			\mathbf{0} & I_{p-p_0}
		\end{pmatrix} = \mathbf{0}_{(n+1)p\times (m-\Card A+p-p_0)},
	\end{equation}
	where 
	\begin{align*} 
		J&=:
		\begin{pmatrix} J_1 \\ J_2 \end{pmatrix}
		\in 
		\begin{pmatrix}
			M_{np\times (m-\Card A)}(\RR)\\
			M_{(p_0-m+\Card A)\times (m-\Card A)}(\RR)
		\end{pmatrix},\\
		\widehat{G}&=:
		\begin{pmatrix} \widehat{G}_1 \\ \widehat{G}_2 \\ \widehat{G}_3 \end{pmatrix}
		\in 
		\begin{pmatrix}
			M_{np\times (p-p_0)}(\RR)\\
			M_{(p_0-m+\Card A)\times (p-p_0)}(\RR)\\
			M_{(m-\Card A)\times (p-p_0)}(\RR)
		\end{pmatrix}.
	\end{align*}
	\smallskip
	
	\noindent\textit{Proof of Claim 3.}
	By \eqref{rank-equality-v11-semidefinite-case} and \eqref{rank-equality-v2-semidefinite-case},
	to prove \eqref{eq:to-prove-280925}
	it suffices to establish the equality 
	\begin{equation}
		\label{eq:definition-of-widehatZ-semidefinite-case}
		\begin{pmatrix}
			\widehat{K}^T & \widehat{Z}
		\end{pmatrix} 
		\begin{pmatrix}
			-J_1 & -\widehat{G}_1\\
			-J_2 & -\widehat{G}_2 \\
			I_{m-\Card A} & -\widehat{G}_3 \\
			\mathbf{0} & I_{p-p_0}
		\end{pmatrix} = 
		\mathbf{0}_{p\times (m-\Card A+p-p_0)}.
	\end{equation}
	Let us write 
	\begin{equation}
		\label{def:Z-semidefinite-case}
		\widehat{Z} 
		:= 
		\begin{pmatrix}
			\widehat{Z}_1 & \widehat{Z}_2 & \widehat{Z}_3 \\[0.2em]
			\widehat{Z}_2^T & \widehat{Z}_4 & \widehat{Z}_5 \\[0.2em]
			\widehat{Z}_3^T & \widehat{Z}_5^T & \widehat{Z}_6 
		\end{pmatrix},
	\end{equation}
	where 
	$\widehat{Z}_1\in M_{p_0-m+\Card A}(\RR)$, 
	$\widehat{Z}_2\in M_{(p_0-m+\Card A)\times (m-\Card A)}(\RR)$, 
	$\widehat{Z}_3\in M_{(p_0-m+\Card A)\times (p-p_0)}(\RR)$, 
	$\widehat{Z}_4\in M_{m-\Card A}(\RR)$, 
	$\widehat{Z}_5\in M_{(m-\Card A)\times (p-p_0)}(\RR)$ and $\widehat{Z}_6\in M_{p-p_0}(\RR)$.
	In this notation, \eqref{eq:definition-of-widehatZ-semidefinite-case} becomes
	\begin{equation*}
		\begin{pmatrix}
			\widehat{K}_1^T & \widehat{Z}_1 & \widehat{Z}_2 & \widehat{Z}_3\\[0.2em]
			\widehat{K}_2^T & \widehat{Z}_2^T & \widehat{Z}_4 & \widehat{Z}_5\\[0.2em]
			\widehat{K}_3^T & \widehat{Z}_3^T & \widehat{Z}_5^T & \widehat{Z}_6\\
		\end{pmatrix} 
		\begin{pmatrix}
			-J_1 & -\widehat{G}_1\\
			-J_2 & -\widehat{G}_2 \\
			I_{m-\Card A} & -\widehat{G}_3 \\
			\mathbf{0} & I_{p-p_0}
		\end{pmatrix}
		= \mathbf{0}_{p\times (m-\Card A+p-p_0)}.
	\end{equation*}
	We choose $\widehat{Z}_1 \in \Sym_{p_0-m+\Card A}(\RR)$ such that
	\begin{equation}
		\label{choice-of-hat-Z1-semidefinite-case}
		\Rank
		\begin{pmatrix}
			\cH & \widehat{K}_1\\
			\widehat{K}_1^T & \widehat{Z}_1
		\end{pmatrix}
		=
		\Rank 
		\begin{pmatrix} 
			\mathcal{H} \\
			\widehat{K}_1^T
		\end{pmatrix} 
		+ p_0-m+\Card A
	\end{equation}
	and define 
	\begin{equation}
		\begin{alignedat}{2}
			\label{equality-v3-semidefinite-case}
			\widehat{Z}_2 &:= \begin{pmatrix}
				\widehat{K}_1^T & \widehat{Z}_1
			\end{pmatrix}J,\qquad 
			&\widehat{Z}_3 &:= \begin{pmatrix}
				\widehat{K}_1^T & \widehat{Z}_1 & \widehat{Z}_2
			\end{pmatrix}\widehat{G},\\
			\widehat{Z}_4 &:= \begin{pmatrix}
				\widehat{K}_2^T & \widehat{Z}_2^T
			\end{pmatrix}J,\quad 
			&\widehat{Z}_5 &:= \begin{pmatrix}
				\widehat{K}_2^T & \widehat{Z}_2^T & \widehat{Z}_4
			\end{pmatrix}\widehat{G},\\
			&&\widehat{Z}_6 &:= \begin{pmatrix}
				\widehat{K}_3^T & \widehat{Z}_3^T & \widehat{Z}_5^T
			\end{pmatrix}\widehat{G}.
		\end{alignedat}
	\end{equation}
	Once we prove \eqref{eq:definition-of-widehatZ-semidefinite-case},
	\eqref{choice-of-hat-Z1-semidefinite-case} will also imply
	\eqref{choice-of-hat-Z-semidefinite-case}.
	
	Let us check that $\widehat{Z}$ is symmetric.
	Since $\widehat{Z}_1\in \Sym_{p_0-m+\Card A}(\RR)$, we only need to show that $\widehat{Z}_4\in \Sym_{m-\Card A}(\RR)$ and  $\widehat{Z}_6\in \Sym_{p-p_0}(\RR)$. But this follows by the following computations:
	\begin{align*} 
		\widehat{Z}_4 
		&\underbrace{=}_{\eqref{equality-v3-semidefinite-case}} \begin{pmatrix}
			\widehat{K}_2^T & \widehat{Z}_2^T
		\end{pmatrix}J= \begin{pmatrix}
			\widehat{K}_2 \\
			\widehat{Z}_2
		\end{pmatrix}^T  J 
		\underbrace{=}_{
			\substack{\eqref{definition-hat-K-2-psd},\\\eqref{equality-v3-semidefinite-case}}} \begin{pmatrix}
			\begin{pmatrix}
				\mathcal{H} & \widehat{K}_1\\
				\widehat{K}_1^T & \widehat{Z}_1
			\end{pmatrix} J
		\end{pmatrix}^T J 
		= J^T \begin{pmatrix}
			\mathcal{H} & \widehat{K}_1 \\[0.2em]
			\widehat{K}_1^T & \widehat{Z}_1
		\end{pmatrix}^T J, \\
		\widehat{Z}_6 
		&\underbrace{=}_{\eqref{equality-v3-semidefinite-case}} \begin{pmatrix}
			\widehat{K}_3^T & \widehat{Z}_3^T & \widehat{Z}_5^T
		\end{pmatrix}\widehat{G} = \begin{pmatrix}
			\widehat{K}_3 \\
			\widehat{Z}_3 \\
			\widehat{Z}_5
		\end{pmatrix}^T \widehat{G} 
		\underbrace{=}_{\substack{\eqref{hat-K-3-psd-case},\\\eqref{equality-v3-semidefinite-case}}} 
		\begin{pmatrix}
			\begin{pmatrix}
				\mathcal{H} & \widehat{K}_1 & \widehat{K}_2 \\[0.2em]
				\widehat{K}_1^T & \widehat{Z}_1 & \widehat{Z}_2 \\[0.2em]
				\widehat{K}_2^T & \widehat{Z}_2^T & \widehat{Z}_4
			\end{pmatrix}\widehat{G}
		\end{pmatrix}^T \widehat{G}\\
		&= \widehat{G}^T \begin{pmatrix}
			\mathcal{H} & \widehat{K}_1 & \widehat{K}_2 \\
			\widehat{K}_1^T & \widehat{Z}_1 & \widehat{Z}_2 \\
			\widehat{K}_2^T & \widehat{Z}_2^T & \widehat{Z}_4
		\end{pmatrix} \widehat{G}.
	\end{align*}            
	
	We now observe the following:
	\begin{align}
		\label{widehatZ5Tequality}
		\begin{split}
			\widehat{Z}_5^T 
			&\underbrace{=}_{\eqref{equality-v3-semidefinite-case}}
			\left (\begin{pmatrix}
				\widehat{K}_2^T & \widehat{Z}_2^T & \widehat{Z}_4
			\end{pmatrix}\widehat{G}\right )^T=
			\widehat{G}^T\begin{pmatrix}
				\widehat{K}_2 \\
				\widehat{Z}_2 \\
				\widehat{Z}_4
			\end{pmatrix} 
			\underbrace{=}_{
				\substack{\eqref{definition-hat-K-2-psd},\\\eqref{equality-v3-semidefinite-case}}}
			\widehat{G}^T\begin{pmatrix}
				\mathcal{H} & \widehat{K}_1 \\
				\widehat{K}_1^T & \widehat{Z}_1 \\
				\widehat{K}_2^T & \widehat{Z}_2^T \\
			\end{pmatrix}J \\
			&= \left (\begin{pmatrix}
				\mathcal{H} & \widehat{K}_1 & \widehat{K}_2 \\
				\widehat{K}_1^T & \widehat{Z}_1 & \widehat{Z}_2
			\end{pmatrix}\widehat{G}\right )^T J 
			\underbrace{=}_{\substack{\eqref{hat-K-3-psd-case},\\\eqref{equality-v3-semidefinite-case}}} 
			\begin{pmatrix}
				\widehat{K}_3 \\
				\widehat{Z}_3
			\end{pmatrix}^T J =
			\begin{pmatrix}
				\widehat{K}_3^T & \widehat{Z}_3^T
			\end{pmatrix}J.
		\end{split}
	\end{align}
	From the definitions of $\widehat{Z}_i$, $i\in [2;6]$, 
	and the equality \eqref{widehatZ5Tequality}, 
	it is clear that the matrix $\widehat{Z}$ satisfies \eqref{eq:definition-of-widehatZ-semidefinite-case}.
	This proves Claim 3.\hfill$\blacksquare$\\
	
	Define
	the vectors $c_1, c_2, \ldots, c_{m-\Card A+p-p_0}$
	by 
	$$
	C \equiv 
	\row(c_i)_{i\in [m-\Card A+p-p_0]}
	:= (I_{np}\oplus P)
	\begin{pmatrix}
		-J_1 & -\widehat{G}_1\\
		-J_2 & -\widehat{G}_2 \\
		I_{m-\Card A} & -\widehat{G}_3 \\
		\mathbf{0} & I_{p-p_0}
	\end{pmatrix},
	$$
	Note that $c_1, c_2, \ldots, c_{m-\Card A+p-p_0}$ are linearly independent.
	Next define the matrix $Z$ by 
	$$Z:=P\widehat ZP^{T}\in \Sym_p(\RR).$$
	\smallskip
	
	\noindent \textbf{Claim 4.}
	$
	c_1, c_2, \ldots, c_{m-\Card A+p-p_0}
	\in 
	\Ker
	\begin{pmatrix}
		\mathcal{H} & K \\
		K^T & Z
	\end{pmatrix}.
	$\\
	
	\noindent \textit{Proof of Claim 4.}
	We have
	\begin{align*}
		\begin{pmatrix}
			\mathcal{H} & K \\
			K^T & Z
		\end{pmatrix}C &= \begin{pmatrix}
			\mathcal{H} & \widehat{K}P^T \\
			P\widehat{K}^T & P\widehat{Z}P^T
		\end{pmatrix}(I_{np}\oplus P)
		\begin{pmatrix}
			-J_1 & -\widehat{G}_1\\
			-J_2 & -\widehat{G}_2 \\
			I_{m-\Card A} & -\widehat{G}_3 \\
			\mathbf{0} & I_{p-p_0}
		\end{pmatrix}
		\\ &= (I_{np}\oplus P)\begin{pmatrix}
			\mathcal{H} & \widehat{K} \\
			\widehat{K}^T & \widehat{Z}
		\end{pmatrix}(I_{np}\oplus P^T)(I_{np}\oplus P)
		\begin{pmatrix}
			-J_1 & -\widehat{G}_1\\
			-J_2 & -\widehat{G}_2 \\
			I_{m-\Card A} & -\widehat{G}_3 \\
			\mathbf{0} & I_{p-p_0}
		\end{pmatrix}
		\\ &= (I_{np}\oplus P)\begin{pmatrix}
			\mathcal{H} & \widehat{K} \\
			\widehat{K}^T & \widehat{Z}
		\end{pmatrix}
		\begin{pmatrix}
			-J_1 & -\widehat{G}_1\\
			-J_2 & -\widehat{G}_2 \\
			I_{m-\Card A} & -\widehat{G}_3 \\
			\mathbf{0} & I_{p-p_0}
		\end{pmatrix}
		\\ &\underbrace{=}_{\eqref{eq:to-prove-280925}} (I_{np}\oplus P)\mathbf{0}_{(n+1)p\times (m-\Card A+p-p_0)} = \mathbf{0}_{(n+1)p\times (m-\Card A+p-p_0)}.
	\end{align*}
	This proves Claim 4.\hfill$\blacksquare$\\
	
	Define
	\begin{align}
		\label{def-widehatTn}
		\widehat{T}_n \equiv 
		\big(
		\underbrace{\widehat{T}_{n,1}}_{\substack{p_0-m+\Card A\\\text{columns}}} \; \underbrace{\widehat{T}_{n,2}}_{\substack{m-\Card A\\\text{columns}}} \; 
		\underbrace{\widehat{T}_{n,3}}_{\substack{p-p_0\\\text{columns}}}
		\big) := T_nP
		=T_nQ(P_1 \oplus I_{p-p_0}).
	\end{align}
	
	\noindent \textbf{Claim 5.} Recall the notation in \eqref{def-mathcalH0-etc}. We have
	\begin{align}
		\label{K-Z-columns-dependancy}
		\Rank \begin{pmatrix} 
			\mathcal{H}_0 & \mathcal{H} & \widehat{K} \\
			\widehat{T}_n^T & \widehat{K}^T & \widehat{Z}
		\end{pmatrix} = \Rank \begin{pmatrix} 
			\mathcal{H}_0 & \mathcal{H} \\
			\widehat{T}_n^T & \widehat{K}^T
		\end{pmatrix}
		=
		\Rank M_{\cT}(n).
	\end{align}
	
	\noindent \textit{Proof of Claim 5.}
	Since 
	$$\begin{pmatrix} 
		\mathcal{H}_0 & \mathcal{H} \\
		\widehat{T}_n^T & \widehat{K}^T
	\end{pmatrix} = (I_{np} \oplus P^T)\begin{pmatrix} 
		\mathcal{H}_0 & \mathcal{H} \\
		T_n & K^T
	\end{pmatrix} = (I_{np} \oplus P^T)M_{\cT}(n),$$
	the nontrivial equality in \eqref{K-Z-columns-dependancy}
	is the first one.
	Observe that
	\begin{align}
		\label{K-Z-columns-dependancy-v2}
		\begin{split}
			&\begin{pmatrix} 
				\mathcal{H}_0^T & \widehat{T}_n \\
				\mathcal{H} &\widehat{K}
			\end{pmatrix} \\
			&= M_{\cT}(n)(I_{np} \oplus P) \\
				&= \begin{pmatrix}
					\row({X}_{\cT}^{i})_{i\in [0;n-1]} & 
					\widecheck{X}_{\cT}^{n}(P_1 \oplus I_{p-p_0})
				\end{pmatrix} \\
				&= 
				\begin{pmatrix}
					\row({X}_{\cT}^{i})_{i\in [0;n-1]} & 
					(\widecheck{X}_{\cT}^{n})_{(0)}P_1 &
					\row\big(
					(\widecheck{X}_{\cT}^{n})_{(i,0)}\;
					(\widecheck{X}_{\cT}^{n})_{(i,1)}
					\big)_{i\in [s]}
				\end{pmatrix} \\
				&=
				\Big(
				\underbrace{\begin{pmatrix} 
						\row({X}_{\cT}^{i})_{i\in [0;n-1]} & (\widecheck{X}_{\cT}^{n})_{(0)}
					\end{pmatrix}(I_{np} \oplus P_1)}_{= \begin{pmatrix}
						\mathcal{H}_0^T & \widehat{T}_{n,1} &
						\widehat{T}_{n,2} \\
						\mathcal{H} & \widehat{K}_1 & \widehat{K}_2
				\end{pmatrix}} \quad \; 
				\underbrace{\row\big(
					(\widecheck{X}_{\cT}^{n})_{(i,0)}\;
					(\widecheck{X}_{\cT}^{n})_{(i,1)}
					\big)_{i\in [s]}}_{= \begin{pmatrix}
						\widehat{T}_{n,3} \\
						\widehat{K}_3
				\end{pmatrix}}
				\Big).
			\end{split}
		\end{align}
		By \eqref{column-relation-v6-semidefinite-case}, 
		\eqref{def-widehatG} and \eqref{K-Z-columns-dependancy-v2}, 
		it follows that
		\begin{align}
			\label{column-relation-v7-semidefinite-case}
			\widehat{T}_{n,3} = 
			\begin{pmatrix}
				\mathcal{H}_0^T 
				& \widehat{T}_{n,1} 
				& \widehat{T}_{n,2} 
			\end{pmatrix}\widehat{G}.
		\end{align}
		We have
		\begin{align}
			\label{auxiliary-rank-equalities}
			\begin{split}
				&\Rank \begin{pmatrix} 
					\mathcal{H}_0 & \mathcal{H} & \widehat{K} \\
					\widehat{T}_n^T & \widehat{K}^T & \widehat{Z}
				\end{pmatrix}
				\underbrace{=}_{(\ast)^T}
				\Rank \begin{pmatrix} 
					\mathcal{H}_0^T & \widehat{T}_n\\
					\mathcal{H} & \widehat{K} \\
					\widehat{K}^T & \widehat{Z}
				\end{pmatrix}\\
				=&
				\Rank \begin{pmatrix} 
					\mathcal{H}_0^T & \widehat{T}_{n,1} & 
					\widehat{T}_{n,2} & \widehat{T}_{n,3}\\
					\mathcal{H} & \widehat{K}_1 & \widehat{K}_2 & \widehat{K}_3\\
					\widehat{K}_1^T & \widehat{Z}_{1} & \widehat{Z}_{2} & \widehat{Z}_{3}\\
					\widehat{K}_2^T & \widehat{Z}_{2}^T & \widehat{Z}_{4} & \widehat{Z}_{5}\\
					\widehat{K}_3^T & \widehat{Z}_{3}^T & \widehat{Z}_{5}^T & \widehat{Z}_{6}
				\end{pmatrix}
				\underbrace{=}_{\substack{
						\eqref{hat-K-3-psd-case},\\
						\eqref{equality-v3-semidefinite-case},
						\\\eqref{column-relation-v7-semidefinite-case}}}
				\Rank \begin{pmatrix} 
					\mathcal{H}_0^T & \widehat{T}_{n,1} & 
					\widehat{T}_{n,2} \\
					\mathcal{H} & \widehat{K}_1 & \widehat{K}_2 \\
					\widehat{K}_1^T & \widehat{Z}_{1} & \widehat{Z}_{2} \\
					\widehat{K}_2^T & \widehat{Z}_{2}^T & \widehat{Z}_{4} \\
					\widehat{K}_3^T & \widehat{Z}_{3}^T & \widehat{Z}_{5}^T 
				\end{pmatrix}\\ 
				\underbrace{=}_{(\ast)^T}&      
				\Rank \begin{pmatrix} 
					\mathcal{H}_0 & \mathcal{H} & \widehat{K}_1 & \widehat{K}_2         
					&\widehat{K}_3\\[0.2em]
					\widehat{T}_{n,1}^T & \widehat{K}_1^T & \widehat{Z}_1 & 
					\widehat{Z}_2 & \widehat{Z}_3\\[0.2em]
					\widehat{T}_{n,2}^T & \widehat{K}_2^T & \widehat{Z}_2^T &       
					\widehat{Z}_4 & \widehat{Z}_5\\
				\end{pmatrix} 
				\underbrace{=}_{\substack{\eqref{hat-K-3-psd-case},\\
						\eqref{definition-hat-K-2-psd},\\
						\eqref{equality-v3-semidefinite-case}
				}}
				\Rank \begin{pmatrix} 
					\mathcal{H}_0 & \mathcal{H} & \widehat{K}_1 \\[0.2em]
					\widehat{T}_{n,1}^T & \widehat{K}_1^T & \widehat{Z}_1\\[0.2em]
					\widehat{T}_{n,2}^T & \widehat{K}_2^T & \widehat{Z}_2^T
				\end{pmatrix}.
			\end{split}
		\end{align}
		To prove \eqref{K-Z-columns-dependancy}, \eqref{auxiliary-rank-equalities}
		implies that 
		it suffices to show that
		\begin{align}
			\label{K1-Z1-Z2T-columns-dependancy}
			\Rank \begin{pmatrix} 
				\mathcal{H}_0 & \mathcal{H} & \widehat{K}_1 \\[0.2em]
				\widehat{T}_{n,1}^T & \widehat{K}_1^T & \widehat{Z}_1\\[0.2em]
				\widehat{T}_{n,2}^T & \widehat{K}_2^T & \widehat{Z}_2^T
			\end{pmatrix} 
			=
			\Rank M_{\cT}(n).
		\end{align}
		By \eqref{hat-K-3-psd-case}, \eqref{column-relation-v7-semidefinite-case} \and the ordering of the columns (see \eqref{order-permutation-Q}) of the matrix
		$$
		\begin{pmatrix}
			\cH_0^T & \widehat T_{n,1} & \widehat T_{n,2}\\
			\cH & \widehat K_1 & \widehat{K}_2
		\end{pmatrix},
		$$
		note that
		\begin{equation}
			\label{290925-1857}
			\Rank M_{\cT}(n)
			=
			\Bigg(
			\begin{array}{c}
				\cH_{0}^T \\
				\cH
			\end{array}
			\;
			\underbrace{\begin{array}{cc}
					\widehat T_{n,1} & \widehat T_{n,2}\\ 
					\widehat K_1 & \widehat{K}_2
			\end{array}}_{p_0\text{ columns}}
			\Bigg)               
			=
			\Rank\begin{pmatrix}
				\cH_0^T \\
				\cH 
			\end{pmatrix}
			+p_0.
		\end{equation}
		We permute the columns of 
		$\begin{pmatrix}
			\cH_0^T \\
			\cH 
		\end{pmatrix}
		$
		to
		$
		\begin{pmatrix}
			\cH_0^T \\
			\cH
		\end{pmatrix}P_2=:
		\begin{pmatrix}
			\cH_{0,1}^T & \cH_{0,2}^T\\
			\cH_1^T & \cH_2^T
		\end{pmatrix},
		$
		using a permutation matrix $P_2 \in M_{np}(\RR)$,
		such that
		\begin{equation}
			\label{reordering-290925}
			\Rank 
			\begin{pmatrix}
				\cH_{0,1} & \cH_1 \\ 
				\cH_{0,2} & \cH_2
			\end{pmatrix}
			\underbrace{=}_{(\ast)^T}
			\Rank
			\begin{pmatrix}
				\cH_0^T \\ \cH^T
			\end{pmatrix}
			=
			\Rank
			\begin{pmatrix}
				\cH_{0,2}^T \\ \cH_{2}^T
			\end{pmatrix}
			\underbrace{=}_{(\ast)^T}
			\Rank            
			\begin{pmatrix} 
				\cH_{0,2} & \cH_2
			\end{pmatrix}
		\end{equation}
		and
		\begin{equation}
			\label{reordering-290925-v2}
			\Ker 
			\begin{pmatrix}
				\cH_{0,2}^T & \widehat T_{n,1} & \widehat T_{n,2}\\
				\cH_2^T & 
				\widehat K_1 & \widehat{K}_2
			\end{pmatrix}
			\underbrace{=}_{\eqref{290925-1857}}\{\mathbf{0}\}.
		\end{equation}
		Apply the same permutation $P_2$          
		on $\widehat{K}_1^T$
		to obtain
		$\widehat{K}_1^TP_2=:\begin{pmatrix}
			(\widehat{K}_1^{(1)})^T & (\widehat{K}_1^{(2)})^T
		\end{pmatrix}$.
		It follows from \eqref{reordering-290925-v2} that 
		$$
		\cC
		\begin{pmatrix} 
			\mathcal{H}_{0,2} & \mathcal{H}_2 \\[0.2em]
			\widehat{T}_{n,1}^T & \widehat{K}_1^T \\[0.2em]
			\widehat{T}_{n,2}^T & \widehat{K}_2^T
		\end{pmatrix}=\RR^{\Rank M_{\cT}(n)}
		\quad\text{and}\quad
		\begin{pmatrix} 
			\mathcal{H}_{0,2} & \mathcal{H}_2 \\[0.2em]
			\widehat{T}_{n,1}^T & \widehat{K}_1^T \\[0.2em]
			\widehat{T}_{n,2}^T & \widehat{K}_2^T
		\end{pmatrix}
		\text{ is surjective}, 
		$$ 
		whence
		\begin{align}
			\label{rank-equality-v4-semidefinite-case}
			\Rank \begin{pmatrix} 
				\mathcal{H}_{0,2} & \mathcal{H}_2 &  \widehat{K}_1^{(2)} \\[0.2em]
				\widehat{T}_{n,1}^T & \widehat{K}_1^T & \widehat{Z}_1 \\[0.2em]
				\widehat{T}_{n,2}^T & \widehat{K}_2^T & \widehat{Z}_2^T
			\end{pmatrix} = \Rank \begin{pmatrix} 
				\mathcal{H}_{0,2} & \mathcal{H}_2 \\[0.2em]
				\widehat{T}_{n,1}^T & \widehat{K}_1^T \\[0.2em]
				\widehat{T}_{n,2}^T & \widehat{K}_2^T
			\end{pmatrix}.
		\end{align}
		Since $\cT$ admits a representing measure by assumption in the theorem, there exists a psd extension 
		$$
		\begin{pmatrix} 
			\cH_{0} & \cH & K \\
			T_n & K^T & Z^{(e)}_1\\
			K^T  & (Z^{(e)}_1)^T & Z_{2}^{(e)}
		\end{pmatrix}.
		$$
		By Theorem \ref{block-psd}, used for 
		$A:=
		\begin{pmatrix}
			\cH_0 & \cH \\ T_n & K^T
		\end{pmatrix},$
		$
		B:=
		\begin{pmatrix}
			K\\ Z_1^{(e)}
		\end{pmatrix},
		$
		$C:=Z_2^{(e)}$,
		it follows that $\cC(B)\subseteq \cC(A)$. In particular,
		$\Rank \begin{pmatrix}
			\mathcal{H}_0 & \mathcal{H}
		\end{pmatrix} = \Rank \begin{pmatrix}
			\mathcal{H}_0 & \mathcal{H} & K
		\end{pmatrix}$,
		whence
		\begin{align}
			\begin{split}
				\label{rank-equality-v5-semidefinite-case}
				\Rank \begin{pmatrix} 
					\mathcal{H}_{0,1} & \mathcal{H}_{1} \\[0.2em]
					\mathcal{H}_{0,2} & \mathcal{H}_{2}
				\end{pmatrix} &=
				\Rank \begin{pmatrix} 
					\mathcal{H}_{0,1} & \mathcal{H}_{(1)} & 	\widehat{K}_1^{(1)} \\
					\mathcal{H}_{0,2} & \mathcal{H}_{(2)} & 	\widehat{K}_1^{(2)}
				\end{pmatrix}.
			\end{split}
		\end{align}
		By Lemma \ref{lemma-rank-equalities}, used for
		\begin{align*} 
			A_1&:=
			\begin{pmatrix} 
				\cH_{0,1} & \cH_1 
			\end{pmatrix},
			&B_1&:=
			\widehat K_{1}^{(1)},\\
			A_2&:=
			\begin{pmatrix}
				\cH_{0,2} & \cH_2 
			\end{pmatrix},
			&B_2
			&:=\widehat K_{1}^{(2)},\\
			C&:=
			\begin{pmatrix}
				\widehat{T}_{n,1}^T & \widehat{K}_1^T\\
				\widehat{T}_{n,2}^T & \widehat{K}_2^T
			\end{pmatrix},
			&D&:=
			\begin{pmatrix}
				\widehat{Z}_1\\
				\widehat{Z}_2^T
			\end{pmatrix},
		\end{align*}
		and the rank equalities
		\eqref{reordering-290925},
		\eqref{rank-equality-v4-semidefinite-case} and \eqref{rank-equality-v5-semidefinite-case}, we get
		\begin{align}
			\label{K1-Z1-Z2T-columns-dependancy-proved}
			\begin{split}
				\Rank \begin{pmatrix} 
					\mathcal{H}_{0,1} & \mathcal{H}_1 & \widehat{K}_1^{(1)} \\[0.2em]
					\mathcal{H}_{0,2} & \mathcal{H}_2 & \widehat{K}_1^{(2)} \\[0.2em]
					\widehat{T}_{n,1}^T & \widehat{K}_1^T & \widehat{Z}_1\\[0.2em]
					\widehat{T}_{n,2}^T & \widehat{K}_2^T & \widehat{Z}_2^T
				\end{pmatrix} 
				&= \Rank \begin{pmatrix} 
					\mathcal{H}_{0,1} & \mathcal{H}_1 \\[0.2em]
					\mathcal{H}_{0,2} & \mathcal{H}_2 \\[0.2em]
					\widehat{T}_{n,1}^T & \widehat{K}_1^T \\[0.2em]
					\widehat{T}_{n,2}^T & \widehat{K}_2^T
				\end{pmatrix} 
				\underbrace{=}_{\eqref{290925-1857}}
				\Rank M_{\cT}(n).
			\end{split}
		\end{align}
		Hence, the equality \eqref{K1-Z1-Z2T-columns-dependancy} holds, which concludes the proof of Claim 4. \hfill $\blacksquare$\\
		
		By Claim 5,
		\begin{equation}
			\label{061025-0840}
			\begin{pmatrix}
				\widehat{K}_1 \\
				\widehat{Z}_1 \\
				\widehat{Z}_2^T \\
				\widehat{Z}_3^T
			\end{pmatrix}
			=\begin{pmatrix} 
				\mathcal{H}_0 & \mathcal{H} \\
				\widehat{T}_{n,1}^T & \widehat{K}_1^T \\
				\widehat{T}_{n,2}^T & \widehat{K}_2^T \\
				\widehat{T}_{n,3}^T & \widehat{K}_3^T
			\end{pmatrix}\begin{pmatrix}
				U_1 \\
				U_2
			\end{pmatrix}
		\end{equation}
		for some real matrices $U_1 \in M_{p\times (p_0-m+\Card A)}(\RR)$ and $U_2 \in M_{np\times (p_0-m+\Card A)}(\RR)$ with 
		\begin{equation} 
			\label{rank-of-U1-semidefinite-case}
			\Rank U_1=p_0-m+\Card A \quad (\text{see }\eqref{choice-of-hat-Z1-semidefinite-case}),
		\end{equation}
		or equivalently
		\begin{equation}
			\label{equality-v2-semidefinite-case}
			\begin{pmatrix}
				\mathcal{H}_0 & \mathcal{H} & \widehat{K}_1 \\
				\widehat{T}_{n,1}^T & \widehat{K}_1^T & \widehat{Z}_1 \\
				\widehat{T}_{n,2}^T & \widehat{K}_2^T & \widehat{Z}_2^T \\
				\widehat{T}_{n,3}^T & \widehat{K}_3^T & \widehat{Z}_3^T
			\end{pmatrix}
			\begin{pmatrix}
				-U_1 \\
				-U_2 \\
				I_{p_0-m+\Card A}
			\end{pmatrix}
			=\mathbf{0}_{(n+1)p\times (p_0-m+\Card A)}.
		\end{equation}
		By 
		\eqref{eq:to-prove-280925}
		and
		\eqref{equality-v2-semidefinite-case},
		we have that
		$$
		\begin{pmatrix}
			\mathcal{H}_0 & \mathcal{H} & \widehat K \\
			\widehat{T}_n^T & \widehat K^T & \widehat Z
		\end{pmatrix} 
		\begin{pmatrix}
			-U_1 & \mathbf{0}_{p\times  (m-\Card A)} & \mathbf{0}_{p\times (p-p_0)}\\
			-U_2 & -J_1 & -\widehat{G}_1 \\
			I_{p_0-m+\Card A} & -J_2 & -\widehat{G}_2 \\
			\mathbf{0} & I_{m-\Card A} & -\widehat{G}_3 \\
			\mathbf{0} & \mathbf{0} & I_{p-p_0}
		\end{pmatrix}
		=\mathbf{0}_{(n+1)p\times  p}.
		$$
		Therefore 
		\begin{equation}
			\label{equality-v4-semidefinite-case}
			\begin{split}
				&\begin{pmatrix}
					\widehat{K} \\
					\widehat{Z}
				\end{pmatrix}
				\widehat{Q}_n
				= 
				\row(\widehat{X}_{\cT}^i)_{i\in [n]}
				\begin{pmatrix}
					U_2 & J_1 & \widehat{G}_1
				\end{pmatrix} + 
				\widehat{X}_{\cT}^0
				\begin{pmatrix}
					U_1 & \mathbf{0}_{p\times  (m-\Card A+p-p_0)} 
				\end{pmatrix},
			\end{split}
		\end{equation} where
		\begin{equation}
			\label{def:Q-n}
			\widehat{Q}_n:=\begin{pmatrix}
				I_{p_0-m+\Card A} & -J_2 & -\widehat{G}_2 \\
				\mathbf{0}
				& I_{m-\Card A} & -\widehat{G}_3 \\
				\mathbf{0}
				& \mathbf{0}
				& I_{p-p_0}
			\end{pmatrix}
		\end{equation}
		and
		\begin{align}
			\label{permutation-semidefinite-case}
			\row(\widehat{X}_{\cT}^i)_{i\in [0;n]}
			&:= (I_{np}\oplus P^T)
			\row({X}_{\cT}^i)_{i\in [0;n]}
			= \begin{pmatrix} 
				\mathcal{H}_0 & \mathcal{H} \\
				\widehat{T}_n^T & \widehat{K}^T
			\end{pmatrix}.
		\end{align}
		Note that 
		$\widehat{G}_1 = 
		\col(H_i\widecheck{H}_n^{-1})_{i\in [0;n-1]}$
		and $\col(
		\widehat{G}_2,  
		\widehat{G}_3)
		= P_1^T\widecheck{H}_{n,(0)}\widecheck{H}_n^{-1}$. Writing 
		\begin{equation}
			\label{061025-0843}
			\begin{pmatrix}
				U_2 & J_1
			\end{pmatrix} =: \begin{pmatrix}
				U_{2,0} & J_{1,0} \\
				U_{2,1} & J_{1,1} \\
				\vdots & \vdots\\
				U_{2,{n-1}} & J_{1,n-1}
			\end{pmatrix},
		\end{equation}
		where $U_{2,i} \in M_{p\times (p_0-m+\Card A)}(\RR)$ and $J_{1,i} \in M_{p\times (m-\Card A)}(\RR)$, 
		and defining
		\begin{equation} 
			\label{def:Qi}
			Q_i := \begin{pmatrix}
				U_{2,i} & J_{1,i} & H_i\widecheck{H}_n^{-1}
			\end{pmatrix} 
			\quad \text{for } i \in [0;n-1],
		\end{equation}
		\eqref{equality-v4-semidefinite-case} is equivalent to
		\begin{equation}
			\label{equality-v5-semidefinite-case}
			\begin{split}
				&\begin{pmatrix}
					\widehat{K} \\
					\widehat{Z}
				\end{pmatrix}
				\widehat{Q}_n
				= \sum_{i=1}^{n}\widehat{X}_{\cT}^iQ_{i-1} + 
				\widehat{X}_{\cT}^0
				\begin{pmatrix}
					U_1 & \mathbf{0}_{p\times (m-\Card A+p-p_0)} 
				\end{pmatrix}.
			\end{split}
		\end{equation}
		Using \eqref{permutation-semidefinite-case} and 
		$$\begin{pmatrix}
			\widehat{K} \\
			\widehat{Z}
		\end{pmatrix} = \begin{pmatrix}
			KP \\
			P^TZP
		\end{pmatrix} = (I_{np} \oplus P^T)
		\begin{pmatrix}
			\col(T_i)_{i\in [n+1;2n]} \\
			Z
		\end{pmatrix}P$$
		in \eqref{equality-v5-semidefinite-case}, we get
		\begin{align}
			\label{equality-v552-semidefinite-case}
			\begin{pmatrix}
				\col(T_i)_{i\in [n+1;2n]} \\
				Z
			\end{pmatrix}Q_n
			= \sum_{i=1}^{n}X_{\cT}^iQ_{i-1} + X_{\cT}^0\begin{pmatrix}
				U_1 & \mathbf{0}_{p\times (m-\Card A+p-p_0)} 
			\end{pmatrix},
		\end{align}
		where 
		\begin{equation} 
			\label{061025-0853}
			Q_n:=P\widehat{Q}_n.
		\end{equation}
		We now define
		$$
		T_{2n+1}:=Z\qquad \text{and}\qquad 
		T_{2n+2}:=
		\row(T_i)_{i\in [n+1;2n+1]}
		(M_\cT(n))^\dagger
		\col(T_i)_{i\in [n+1;2n+1]}.
		$$
		Let 
		\begin{align}
			\label{061025-0900}
			\begin{split}
				\cT^{(e)}
				&:=(T_i)_{i\in [0;2n+2]},\;
				S_{2n+\ell}:=L_{\cT^{(e)}}((x+t)^{2n+\ell}),\; \ell=1,2,
				\quad\text{and}\\
				{\cS}^{(e)}
				&:=(S_i)_{i\in [0;2n+2]}.
			\end{split}
		\end{align}
		{Define a matrix polynomial
			\begin{align}
				\label{gen-poly-semidefinite-case}
				H(x) := x^{n+1}Q_n - \sum_{i=1}^{n}x^iQ_{i-1} - \begin{pmatrix}
					U_1 & \mathbf{0}_{p\times (m-\Card A+p-p_0)} 
				\end{pmatrix},
			\end{align}
			By \eqref{equality-v552-semidefinite-case} and the definition of $\cT^{(e)}$,
			we have that $H(x)$ is a block column relation of $M_{\cT^{(e)}}(n+1)$.
			{By Corollary \ref{co:connection-between-both-mm}, $H(x-t)$  is a block column relation of 
				$M_{\cS^{(e)}}(n+1)$.}
			By \eqref{rank-of-U1-semidefinite-case}, we have 
			\begin{equation}
				\label{300925:rank-equality}
				\Rank 
				\begin{pmatrix}
					U_1 & \mathbf{0}_{p\times (m-\Card A+p-p_0)} 
				\end{pmatrix}= p_0-m+\Card A.
			\end{equation}
			Let 
			$$
			\widetilde W_0:=
			\Ker 
			\begin{pmatrix}
				U_1 & \mathbf{0}_{p\times (m-\Card A+p-p_0)} 
			\end{pmatrix}
			\quad\text{and}\quad
			\widetilde W_i:=
			\widetilde W_0
			\bigcap
			\bigcap_{j=0}^{i-1} \Ker Q_{j}
			\;\text{for }i\in [n].
			$$ 
			Note 
			that each $Q_i$ ends with $H_i\widecheck{H}_n^{-1}$ for $i\in [0;n-1]$ (see \eqref{def:Qi}),
			which has $p-p_0$ columns.
			Hence, the following holds:
			\begin{equation} 
				\label{kernel-inlusion}
				\text{Let } i\in [0;n-1].
				\text{ If }
				v\in \bigcap_{j=0}^{i}\Ker(H_j\widecheck{H}_n^{-1}),
				\text{ then }
				\begin{pmatrix}
					0 \\ v
				\end{pmatrix}
				\in 
				\widetilde W_{i+1}.
			\end{equation}
			Let $\widetilde s_i:=\dim \widetilde W_i$ for $i\in [0;n]$.
			Recall the definition of $s_i$ from Claim 2 and note that $\Ker(H_j\widecheck{H}_n^{-1}) = \Ker H_j$ for each $j$.. 
			By 
			\eqref{kernel-inlusion}, it follows that 
			$\widetilde s_i\geq s_{i-1}$ for $i\in [n]$.
			Therefore
			\begin{align}
				\label{kernel-dimension-sum}
				\begin{split}
					\sum_{i=0}^{n}\widetilde s_i
					&\geq \widetilde s_0+\sum_{i=0}^{n-1}s_i
					\geq m-\Card A+p-p_0\\
					&\hspace{1cm}
					+(n-z_s)(p-p_0) + \sum_{j=1}^{s-1}(z_{j+1}-z_j)(p-p_0-r_{s-j}) + \Card A\\
					&=m + (n+1)p-\big((n+1)p_0 + \sum_{j=1}^{s}z_{s-j+1}p_j\big)\\
					&=m + (n+1)p-\Rank M(n),
				\end{split}
			\end{align}
			where we used 
			\eqref{300925:rank-equality} 
			and 
			\eqref{sum-of-dim-ker-Hi-proposition-psd-case}
			in the second inequality
			and Claim 1 in the last equality.
			By Lemma \ref{le:determinant-of-matrix-valued-polynomial-v2},
			the invertibility of $Q_n$ (see \eqref{def:Q-n} and \eqref{061025-0853})
			and \eqref{kernel-dimension-sum}, it follows that
			\begin{align}
				\label{3009-det-exp}
				\det H(x-t) 
				&= (x-t)^{m+(n+1)p-\Rank M(n)}g(x),
			\end{align}
			where $0\neq g(x) \in \RR[x]$ and $\deg (\det H(x-t))=(n+1)p$ (since $Q_n$ is invertible).
			Therefore $\deg g(x)=\Rank M(n)-m$.
			{
				By Theorem \ref{th:Hamburger-matricial}, there exists a $(\Rank M(n))$--atomic 
				representing matrix measure for $L$ 
				of the form 
				$\mu=\sum_{j=1}^\ell \delta_{x_j}A_j$, 
				where $\ell\in \NN$, $x_j\in \RR$ are pairwise distinct, $A_j\in \Sym_p^{\succeq 0}(\RR)$ and $\sum_{j=1}^\ell \Rank A_j = \Rank M(n)$. 
				By Lemma \ref{lemma:atoms} and \eqref{3009-det-exp},
				it follows that $t = x_{j'}$ for some $j' \in [\ell]
				$, and $\Rank A_{j'} \geq m$. 
				Without loss of generality we can assume that $j'=1$.
				Next, we need to establish the following claim.\\
				
				\noindent\textbf{Claim 6.}
				$\Rank A_{1} = m$. \\

				\noindent\textit{Proof of Claim 6.} 
				Suppose on the contrary that $\Rank A_{1} = m'$ for some
				\begin{equation} 
					\label{contradiction-argument-semidefinite-case}
					m' > m. 
				\end{equation}
				Define $\widetilde{\mu} := \mu - \delta_tA_1$. 
				Let $\widetilde{L}:\RR[x] \to \Sym_p(\RR)$ be a linear operator, defined by
				\begin{align*}
					\widetilde{L}(x^i) \equiv \widetilde{S}_i := \int_{\RR} x^i\; d\widetilde{\mu} \quad \text{for } i \in \NN \cup \{0\}.
				\end{align*}
				Let $\widetilde{\cS} := (\widetilde{S}_i)_{i\in [0;2n+2]}.$
				We have
				\begin{align}
					\label{rank-inequalities-semidefinite-case-v2}
					\begin{split}
						\Rank M(n)-m'
						&\leq \Rank M_{\widetilde{\cS}}(n)\leq \Rank M_{\widetilde{\cS}}(n+1),\\
						\Rank M_{\widetilde{\cS}}(n+1) &\leq \sum_{j=2}^\ell \Rank A_j = \Rank M(n)-m',
					\end{split}
				\end{align}
				where the first inequality follows
				from the fact that the difference
				$M(n)-M_{\widetilde{\cS}}(n)$ is 
				a moment matrix of the linear operator corresponding to the matrix measure $\delta_tA_1$.
				The inequalities
				\eqref{rank-inequalities-semidefinite-case-v2} imply 
				that 
				$
				\Rank M(n)-m'
				\leq M_{\widetilde{\cS}}(n+1)
				\leq 
				\Rank M(n)-m',
				$
				whence all inequalities in \eqref{rank-inequalities-semidefinite-case-v2} must be equalities.
				In particular,
				\begin{align}
					\label{rank-equality-of-widetildeMn-semidefinite-case-v2}
					\Rank M_{\widetilde{\cS}}(n) = \Rank M_{\widetilde{\cS}}(n+1) = \Rank M(n)-m'.
				\end{align}
				For $i\in [0;2n+2]$ we define
				$\widetilde{T}_i := 
				\int_{\RR}(x-t)^id\widetilde{\mu}.$
				Note that 
				\begin{equation}
					\label{identity-widetildeTi-and-Ti-v2}
					\widetilde{T}_i
					=\int_{\RR}(x-t)^id(\widetilde{\mu}+\delta_tA_1)
					=T_i
					\quad
					\text{for }i\in [2n+2].
				\end{equation}
				Define the sequence 
				$\widetilde{\mathcal{T}} := (\widetilde{T}_i)_{i\in [0;2n+2]}$ and
				write
				$$X_{\widetilde{\cT}}^i
				\equiv
				\row\big((x_{\widetilde{\cT}}^{(i)})_j\big)_{j\in [p]}
				:=
				\col(\widetilde{T}_{i+j})_{j\in [0;n+1]}
				$$
				for $i\in [0;n+1]$.
				By \eqref{rank-equality-of-widetildeMn-semidefinite-case-v2} 
				and Proposition \ref{linear transform invariance-nc}.\eqref{linear transform invariance-nc-4}
				we have 
				\begin{align}
					\label{rank-of-matrix-widetildeYi-v2} 
					\begin{split}
						m'
						&=\Rank M_{\cS}(n+1)-
						\Rank\big(\row
						(X^i_{\widetilde{\cT}})_{i\in [0;n+1]}\big)\\
						&=\Rank M_{\cS}(n)-
						\Rank\big(\row
						(X^i_{\widetilde{\cT}})_{i\in [0;n+1]}\big).
					\end{split}
				\end{align}
				By \eqref{identity-widetildeTi-and-Ti-v2} and Claim 1,
				it follows that
				\begin{equation}
					\label{031025-2013}
					\Rank \row(X^{i}_{\widetilde\cT})_{i\in [n]}
					=
					\Rank \row(X^{i}_{\cT})_{i\in [n]}
					=\Rank M_{\cT}(n)-p_0-\Card A.
				\end{equation}
				Further,
				\begin{align}
					\label{031025-2018}
					\begin{split}
						\Rank \row(X^{i}_{\widetilde\cT})_{i\in [n+1]}
						&\underbrace{=}_{\eqref{identity-widetildeTi-and-Ti-v2}} 
						\Rank \row(X^{i}_{\cT})_{i\in [n+1]}\\
						&\underbrace{=}_{\text{Claim }3}
						\Rank \row(X^{i}_{\cT})_{i\in [n]}
						+p_0-m+\Card A\\
						&\underbrace{=}_{\eqref{031025-2013}}
						\Rank M_{\cT}(n)-m.
					\end{split}
				\end{align}
				Finally,
				\begin{equation*}
					\Rank M_{\cT}(n)-m
					\underbrace{=}_{\eqref{031025-2018}}
					\Rank \row(X^{i}_{\widetilde\cT})_{i\in [n+1]}
					\underbrace{\leq}_{\eqref{rank-of-matrix-widetildeYi-v2}}
					\Rank M_{\cT}(n)-m',
				\end{equation*}
				which implies that $m'\leq m$ and contradicts 
				\eqref{contradiction-argument-semidefinite-case}.
				This proves Claim 6.\hfill$\blacksquare$\\
				
				{It remains to establish the moreover part.
					Note that the representing measure is unique if and only if 
					$\widehat Z$, satisfying         
					\eqref{choice-of-hat-Z-semidefinite-case}
					and
					\eqref{eq:to-prove-280925},
					is unique.
					Indeed, in this case the polynomial $H(x)$ as in \eqref{gen-poly-semidefinite-case} is uniquely determined
					and it completely determines the minimal measure for the extended sequence.

					First note that in the decomposition \eqref{def:Z-semidefinite-case} 
					there is no freedom in chosing $\widehat{Z_1}\in \Sym_{p_0-m+\Card A}(\RR)$ 
					such that \eqref{choice-of-hat-Z1-semidefinite-case} holds
					if and only if $p_0-m+\Card A=0$, or equivalently $m=\Card A+p_0$.
					Since $B\subseteq ([p]\setminus \Dep(X_{\cT}^{z_s}))$,
					it follows by definition of $p_0$ (see \eqref{def-of-p0}) that this is further equivalent to 
					$B=[p]\setminus \Dep(X_{\cT}^{z_s})$ and $m=\Card(A\cup B)$,
					which are the first and the second condition in \eqref{cond:uniqueness}. 
					
					By the previous paragraph, the matrices 
					$\widehat{Z}_1$, $\widehat{Z}_2$ and $\widehat{Z}_3$ in \eqref{def:Z-semidefinite-case} are empty. Next, observe the matrices 
					$\widehat{Z}_4$ and $\widehat{Z}_5$ in \eqref{def:Z-semidefinite-case}. 
					By 
					\eqref{equality-v3-semidefinite-case} 
					and
					\eqref{widehatZ5Tequality},
					$\widehat{Z}_4$ and $\widehat{Z}_5$ will be unique if and only if
					$\widehat{K}_2^Tv=\widehat{K}_3^Tv=\mathbf{0}$ for any $v\in \Ker \cH$.
					This is due to the fact that $J$ in \eqref{definition-hat-K-2-psd}
					is unique up to the addition of a matrix with each column from 
					$\Ker \cH$ (note that $\widehat{K}_1$ is an empty matrix).
					Note that 
					$\Ker \cH\subseteq \Ker \widehat K_2^T$
					and
					$\Ker \cH\subseteq \Ker \widehat K_3^T$
					are exactly the third and the fourth condition in \eqref{cond:uniqueness}. 
					
					Finally, we observe the matrix $\widehat{Z}_6$ in \eqref{def:Z-semidefinite-case}. 
					It remains to prove that the four conditions in \eqref{cond:uniqueness} ensure
					$\widehat{Z}_6$ is also unique.
					By 
					\eqref{equality-v3-semidefinite-case},
					it suffices to prove that
					$\begin{pmatrix}
						\widehat{K}_3^T & \widehat{Z}_5^T
					\end{pmatrix}\widehat{G}$
					does not depend on the choice of $\widehat{G}$ satisfying 
					\eqref{rank-equality-v11-semidefinite-case}.
					If $\widehat{G}$ and $\widehat{G}'$
					both satisfy \eqref{rank-equality-v11-semidefinite-case},
					then 
					$$
					\mathbf{0}
					=
					\begin{pmatrix}
						\cH& \widehat{K}_2
					\end{pmatrix}
					(\widehat{G}-\widehat{G}')
					\underbrace{=}_{\eqref{definition-hat-K-2-psd}}
					\cH
					\begin{pmatrix}
						I& J
					\end{pmatrix}
					(\widehat{G}-\widehat{G}'),
					$$
					whence 
					\begin{equation} 
						\label{kernel-inclusion}
						\text{every column of the matrix }
						\begin{pmatrix}
							I& J
						\end{pmatrix}
						(\widehat{G}-\widehat{G}')
						\text{ is in }
						\Ker \cH.
					\end{equation}
					Note that
					$$
					\begin{pmatrix}
						\widehat{K}_3^T & \widehat{Z}_5^T
					\end{pmatrix}
					(\widehat{G}-\widehat{G}')
					\underbrace{=}_{\eqref{widehatZ5Tequality}}
					\widehat{K}_3^T
					\begin{pmatrix}
						I& J
					\end{pmatrix}
					(\widehat{G}-\widehat{G}')
					\underbrace{=}_{
						\substack{
							\eqref{kernel-inclusion},\\
							\Ker\cH\subseteq \Ker\widehat{K}_3^T
						}
					}\mathbf{0},
					$$
					whence 
					$\begin{pmatrix}
						\widehat{K}_3^T & \widehat{Z}_5^T
					\end{pmatrix}\widehat{G}$
					does not depend on the choice of $\widehat{G}$ satisfying 
					\eqref{rank-equality-v11-semidefinite-case}.
					This proves the moreover part of Theorem \ref{th:mainTheorem}.
				}

	\section{Proof of the implication $\eqref{th:mainTheorem-pt1} \Rightarrow \eqref{th:mainTheorem-pt2}$ of Theorem \ref{th:mainTheorem}}
	\label{section:proof-main-reverse}
	
	Let $\mu=\delta_t A_1+\sum_{j=2}^\ell \delta_{x_j}A_j$ be a finitely atomic $p\times p$ matrix representing measure for $L$
	with 
	$\ell\in \NN$, $x_j\in \RR\setminus \{t\}$,
	$A_j\in \Sym_p^{\succeq 0}(\RR)$, 
	$\sum_{j=1}^\ell \Rank A_j = \Rank M(n)$
	and $\mult_t \mu=m$.
	For $k=1,2$ define
	$S_{2n+k} := \int_{\RR} x^{2n+k}\; d\mu$
	and 
	$T_{2n+k} := \int_{\RR} (x-t)^{2n+k}\; d\mu$.
	Let $\widetilde{\mu}:=\sum_{j=2}^\ell \delta_{x_j}A_j$
	and 
	$\widetilde{L}:\RR[x] \to \Sym_p(\RR)$ be a linear operator, defined by
	\begin{align*}
		\widetilde{L}(x^i) \equiv \widetilde{S}_i := \int_{\RR} x^i\; d\widetilde{\mu} \quad \text{for } i \in \NN \cup \{0\}.
	\end{align*}
	Write $\widetilde{\cS} := (\widetilde{S}_i)_{i\in [0;2n+2]}.$
	Analogously as in the proof of Claim 6
	in the proof of the implication 
	$\eqref{th:mainTheorem-pt2} \Rightarrow \eqref{th:mainTheorem-pt1}$ of Theorem \ref{th:mainTheorem},
	it follows that
	\begin{align}
		\label{rank-equality-of-widetildeMn-semidefinite-case}
		\Rank M_{\widetilde{\cS}}(n) = \Rank M_{\widetilde{\cS}}(n+1) = \Rank M(n)-m.
	\end{align}
	For $i\in [0;2n+2]$ we define
	$\widetilde{T}_i := 
	\int_{\RR}(x-t)^id\widetilde{\mu}.$
	Note that 
	\begin{equation*}
		\widetilde{T}_i
		=\int_{\RR}(x-t)^id(\widetilde{\mu}+\delta_tA_1)
		=T_i
		\quad
		\text{for }i\in [2n+2].
	\end{equation*}
	Define the sequence 
	$\widetilde{\mathcal{T}} := (\widetilde{T}_i)_{i\in [0;2n+2]}$ and
	for $i\in [0;n+1]$
	write
	\begin{align*} 
		X_{{\cT}}^i
		\equiv
		\row\big((x_{{\cT}}^{(i)})_j\big)_{j\in [p]}
		:=
		\col({T}_{i+j})_{j\in [0;n+1]},\\
		X_{\widetilde{\cT}}^i
		\equiv
		\row\big((x_{\widetilde{\cT}}^{(i)})_j\big)_{j\in [p]}
		:=
		\col(\widetilde{T}_{i+j})_{j\in [0;n+1]}.
	\end{align*}
	By \eqref{rank-equality-of-widetildeMn-semidefinite-case} 
	and Proposition \ref{linear transform invariance-nc}.\eqref{linear transform invariance-nc-4}
	we have 
	\begin{align*}
		\begin{split}
			m
			&=\Rank M(n)-
			\Rank\big(\row
			(X^i_{\widetilde{\cT}})_{i\in [0;n+1]}\big)\\
			&=\Rank M(n)-
			\Rank\big(\row
			(X^i_{\widetilde{\cT}})_{i\in [0;n]}\big).
		\end{split}
	\end{align*}
	Let $\Dep(X_{\ast}^i)$, $\Dep_1(X_{\ast}^i)$ and $\Depn(X_{\ast}^i)$, $\Depn_{\ell}(X_{\ast}^i)$, $\ell=0,1$,
	be
	as in \eqref{eq:defDep-of-Y-i},
	\eqref{eq:defDep-of-Y-i-alternative} and \eqref{eq:defDepnew-of-Y-i}, 
	respectively.
	Define
	\begin{equation*}
		\mathcal{Z} \equiv \{
		z_1,\ldots,z_s\colon 0\leq z_1<z_2<\ldots<z_s\leq n+1
		\}
		:= \big\{i \in [0;n+1] 
		\colon
		\Depn(X_{\widetilde\cT}^i) \neq \emptyset \big\}
		. 
	\end{equation*}
	Note that $\Dep(X_{\widetilde{\cT}}^{z_s})=[p]$.
	Let
	\begin{align*}
		\begin{split}
			p_{j,\ell} &:=\Card\Depn_\ell(X_\cT^{z_{s
					-j+1}})\;\;\text{for }j\in [s] \text{ and }\ell=0,1,\\
			p_j &:=\Card\Depn(X_\cT^{z_{s-j+1}})=p_{j,0}+p_{j,1}
			\;\;\text{for }j\in [s],
		\end{split}
	\end{align*}
	and 
	$$\widetilde A:=\bigcup_{i=1}^{n+1}\Depn_1(X_{\widetilde{\cT}}^i).$$
	
	\noindent \textbf{Claim 1.}
	We have 
	\begin{align}
		\label{051025-0759}
		\Rank M_{\widetilde{\cT}}(n+1)
		&=\Rank M_{\widetilde{\cT}}(n)
		=\sum_{j=1}^{s}z_{s-j+1}p_j.
	\end{align}
	
	\noindent \textit{Proof of Claim 1.}
	\eqref{051025-0759} follows by definition of $p_j$ and $z_j$, together with
	the fact that
	$M_{\widetilde{\cT}}(n)$ is block--recursively generated, which in particular implies that 
	if $i\in \Dep(X_{\widetilde{\cT}}^{j})$ for some $j\in [0;n]$, then $i\in \Dep(X_{\widetilde{\cT}}^{j+1}).$
	\hfill$\blacksquare$\\

	\noindent \textbf{Claim 2.}
	$\widetilde A=\emptyset$. \\
	
	\noindent\textit{Proof of Claim 2.}
	Analogously as in the proof of Claim 2 of the implication $\eqref{th:mainTheorem-pt2} \Rightarrow \eqref{th:mainTheorem-pt1}$ of Theorem \ref{th:mainTheorem} with $\widetilde{\cT}$ in place of $\cT$ and noting that $p_0=0$, there exists a matrix polynomial 
	$${H}(x) = x^{n+1}H_{n+1} - \sum_{i=0}^{n}x^i{H}_i,$$ 
	such that:
	\begin{enumerate}
		\item $H(x)$ is a block column relation of 
		$M_{\widetilde{\cT}}(n+1)$.
		\item 
		$H_{n+1}$ is invertible.
		\item Defining
		$W_i:=\bigcap_{j=0}^i\Ker H_j$ and $s_i:=\dim W_i$
		for $i\in [0;n]$,
		we have
		\begin{align}
			\label{def:kernel-dimensions}
			\begin{split}
				\sum_{i=0}^{n} s_i 
				&\geq (n+1-z_{s})p + \sum_{j=1}^{s-1}(z_{j+1}-z_j)\big(p-\sum_{k=1}^{s-j} {p}_k\big) + \Card\widetilde{A}\\
				&= (n+1)p- \sum_{j=1}^{s}z_{s-j+1}p_j + \Card\widetilde{A}\\
				&\underbrace{=}_{\text{Claim }1}
				(n+1)p-\Rank M_{\widetilde{\cT}}(n)+
				\Card\widetilde{A}.
			\end{split}
		\end{align}
	\end{enumerate}
	By Lemma \ref{le:determinant-of-matrix-valued-polynomial-v2},
	the invertibility of $H_{n+1}$
	and \eqref{def:kernel-dimensions}, it follows that
	\begin{align*}
		\det H(x) 
		&= 
		x^{(n+1)p-\Rank M_{\widetilde{\cT}}(n)+
			\Card\widetilde{A}}
		g(x),
	\end{align*}
	where $0\neq g(x) \in \RR[x]$ and $\deg (\det H(x))=(n+1)p$ (since $H_{n+1}$ is invertible).
	Therefore $\deg g(x)=\Rank M_{\widetilde{\cT}}(n)-\Card \widetilde{A}$.
	Since 
	$\sum_{j=2}^{\ell}\delta_{x_j-t}A_j$
	is a representing measure for $\widetilde\cT$,
	$\sum_{j=2}^{\ell}\Rank A_j = \Rank M_{\widetilde{\cT}}(n)$
	and $x_j-t\neq 0$ for each $j \in [2;\ell]$,
	it follows by Lemma \ref{lemma:atoms}
	that $x_j-t$ are precisely the zeroes of $g(x)$ 
	and $\Rank A_j$ are their multiplicities.
	Therefore 
	$\deg g(x)=\Rank M_{\widetilde{\cT}}(n),$
	whence
	$\Card\widetilde{A} =0$.  
	\hfill $\blacksquare$\\
	
	
	Let $B$ be as in Theorem \ref{th:mainTheorem}. Note that
	\begin{equation}
		\label{051025-0908}
		\Depn_1(X_{\cT}^{n+1})\subseteq B.
	\end{equation}

	\noindent \textbf{Claim 3.}
	For $i\in [0;n]$ we have
	\begin{equation} 
		\label{equality:250925-1444}
		\Depn_1(X_{{\cT}}^{i+1})
		\sqcup
		{\Dep}(X_{{\cT}}^{i})
		=
		{\Dep}(X_{\widetilde{\cT}}^{i}).
	\end{equation}
	
	\noindent\textit{Proof of Claim 3.}
	First we prove the inclusion $(\subseteq)$ in \eqref{equality:250925-1444}.
	Fix $i\in [0;n]$.
	Since $M_{\cT}(n+1)\succeq M_{\widetilde\cT}(n+1)$,
	Lemma \ref{110722-1428} (used for $A=M_{\cT}(n+1)$,
	$B=M_{\widetilde\cT}(n+1)$) implies that
	\begin{equation} 
		\label{250925-0803}
		\Depn_1(X_{{\cT}}^{i+1})\subseteq
		\Dep(X_{\widetilde{\cT}}^{i+1})
		\qquad
		\text{and}
		\qquad
		{\Dep}(X_{{\cT}}^{i})\subseteq {\Dep}(X_{\widetilde{\cT}}^{i}).
	\end{equation}
	Using Claim 2, 
	$$\Depn_1(X_{{\cT}}^{i+1})\subseteq
	\Dep_1(X_{\widetilde{\cT}}^{i+1})\setminus 
	\Depn_1(X_{\widetilde{\cT}}^{i+1}),
	$$
	whence 
	\begin{equation}
		\label{051025-0912}
		\Depn_1(X_{{\cT}}^{i+1})\subseteq
		{\Dep}(X_{\widetilde{\cT}}^{i}).
	\end{equation}
	By \eqref{051025-0912} and \eqref{250925-0803},
	the inclusion $(\subseteq)$ in \eqref{equality:250925-1444} follows.
	
	It remains to prove the inclusion
	$(\supseteq)$ in \eqref{equality:250925-1444}.
	Fix $i\in [0;n]$.
	Let $j\in {\Dep}(X_{\widetilde{\cT}}^{i})$. We separate two cases.\\
	
	\noindent\textbf{Case 1:}
	$j\in \Depn_1(X_{{\cT}}^{i})$.
	Since 
	$\Depn_1(X_{{\cT}}^{i})\subseteq \Dep(X_{{\cT}}^{i})$,
	it follows that 
	$j\in \Dep(X_{{\cT}}^{i})$.\\
	
	\noindent\textbf{Case 2:}
	$j\notin \Depn_1(X_{{\cT}}^{i})$.
	Assume also that $j\notin {\Dep}(X_{{\cT}}^{i})$.
	Since $M_{\widetilde{\cT}}(n+1)$ is block--recursively generated, it follows that 
	$j\in \Dep_1(X_{\widetilde{\cT}}^{i+1})$
	and hence the equalities
	$X^i_{{\cT}}=X^i_{\widetilde{\cT}}$ 
	for $i\in [n+1]$ imply  
	$j\in \Depn_1(X_{{\cT}}^{i+1})$.\\
	
	This proves the inclusion
	$(\supseteq)$ in \eqref{equality:250925-1444}
	\hfill$\blacksquare$.\\
	
	
	We have
	\begin{align*}
		m&=\Rank M_{\cT}(n)-\Rank M_{\widetilde\cT}(n)\\
		&=\Big((n+1)p-\sum_{i=0}^{n}\Card\Big(\Dep(X^i_{{\cT}})\Big)\Big)-
		\Big((n+1)p-\sum_{i=0}^{n}\Card\Big(\Dep(X^i_{\widetilde{\cT}})\Big)\Big)\\
		&\underbrace{=}_{\text{Claim }3}\sum_{i=0}^{n}\Card\Big(\Depn_1(X^{i+1}_{{\cT}})\Big)\\
		&=\Card A+\Card\Big(\Depn_1(X^{n+1}_{{\cT}})\Big)
		\underbrace{\leq}_{\eqref{051025-0908}}
		\Card(A\cup B).
	\end{align*}
	Hence,
	$\Card(A)\leq m\leq \Card(A\cup B)$,
	which proves
	the implication $\eqref{th:mainTheorem-pt1} \Rightarrow \eqref{th:mainTheorem-pt2}$ of Theorem \ref{th:mainTheorem}.

	\section{Proof of Theorem \ref{co:corollary-atom-avoidance}}
	\label{section:proof-theorem-multiplicity-0}
	
	The equivalence \eqref{co:corollary-atom-avoidance-pt1} $\Leftrightarrow$ \eqref{co:corollary-atom-avoidance-pt2} follows directly from Theorem \ref{th:mainTheorem} for $m=0$.
	
	Let us prove the implication \eqref{co:corollary-atom-avoidance-pt1} $\Rightarrow$ \eqref{co:corollary-atom-avoidance-pt3}.
	Let $\mu$ be a minimal representing measure for $L$ such that $\mult_\mu t=0$
	and let $L((x-t)^{i}):=\int_\RR (x-t)^{i} d\mu$, 
	$i\in \{-2,-1,2n+1,2n+2\}$,
	be the extension of $L$ to 
	$\Span\{(x-t)^{-2},(x-t)^{-1},1,x-t,\ldots,(x-t)^{2n+1},(x-t)^{2n+2}\}.$
	Write $T_i:=L((x-t)^i)$ for $i\in [-2;2n+2]$.
	Since $(T_i)_{i=-2}^{2n+2}$ has a representing measure,
	it follows that the matrix
	$\widetilde{\cM}_1:=(T_{i+j-2})_{i,j=0}^{n+2}$
	is positive semidefinite.
	
	Note that
	$$
	\widetilde \cM_1
	=
	\begin{pmatrix}
		T_{-2} & \ast & \ast\\
		\ast & \cM_2 & \ast\\
		\underbrace{\ast}_{p\text{ columns}} & \underbrace{\ast}_{np\text{ columns}} & \underbrace{\ast}_{2p\text{ columns}}
	\end{pmatrix}
	=
	\begin{pmatrix}
		\ast & \ast & \ast\\
		\underbrace{\ast}_{p\text{ columns}} & \underbrace{\cM_3}_{np\text{ columns}} & \underbrace{\ast}_{2p\text{ columns}}
	\end{pmatrix}.
	$$
	Let $w\in \Ker \cM_2$. By Lemma \ref{extension-principle}
	(used for $A=\widetilde{\cM}_1\in \Sym_{(n+3)p}(\RR)$ and
	$A|_{Q}=\cM_2$ where $Q=[p+1;(n+1)p]$),
	it follows that 
	$\col( 
	\mathbf{0}_{p\times 1} , w , \mathbf{0}_{2p\times 1})
	\in \Ker \widetilde{\cM}_1,$
	whence $w\in \Ker \cM_3$.
	Therefore, $\Ker \cM_2\subseteq \Ker \cM_3$.
	
	Similarly 
	$$
	\widetilde \cM_1
	=
	\begin{pmatrix}
		\ast & \ast & \ast\\
		\ast & \cM_3 & \ast\\
		\underbrace{\ast}_{2p\text{ columns}} & \underbrace{\ast}_{np\text{ columns}} & \underbrace{T_{2n+2}}_{p\text{ columns}}
	\end{pmatrix}
	=
	\begin{pmatrix}
		\ast & \cM_2 & \ast\\
		\underbrace{\ast}_{2p\text{ columns}} & \underbrace{\ast}_{np\text{ columns}} & \underbrace{\ast}_{p\text{ columns}}
	\end{pmatrix}.
	$$
	Let $w\in \Ker \cM_3$. By Lemma \ref{extension-principle}
	(used for $A=\widetilde{\cM}_1\in \Sym_{(n+3)p}(\RR)$ and
	$A|_{Q}=\cM_3$ where $Q=[2p+1;(n+2)p]$),
	it follows that 
	$\col( 
	\mathbf{0}_{2p\times 1} , w , \mathbf{0}_{p\times 1})
	\in \Ker \widetilde{\cM}_1,$
	whence $w\in \Ker \cM_2$.
	Therefore, $\Ker \cM_3\subseteq \Ker \cM_2$.
	
	It remains to prove the implication \eqref{co:corollary-atom-avoidance-pt3} $\Rightarrow$ \eqref{co:corollary-atom-avoidance-pt2}.
	We have to prove that $A=\emptyset$. Assume on the contrary that 
	there exists $i\in A$. By definition of $A$, there is $k\in [n]$ such 
	that $\mathbf{y}_{kp+i}=\sum_{\ell=p+1}^{kp+i-1}\gamma_\ell \mathbf{y}_\ell$ for some $\gamma_\ell\in \RR$
	and in particular,
	$\mathbf{y}_{(k-1)p+i}\neq \sum_{\ell=1}^{(k-1)p+i-1}\gamma_{\ell +p}\mathbf{y}_\ell$. But this implies
	$$w=(-\gamma_{p+1},\ldots,-\gamma_{kp+i-1},1,0,\ldots,0)\in \RR^{np}\in \Ker \cM_3\setminus \Ker \cM_2,$$
	which is contradiction to $\Ker\cM_2=\Ker\cM_3$. Hence, $A=\emptyset$.

	\section{Proof of Corollary \ref{co:Simonov-psd}}
	\label{section:corollary}
	The equivalence 
	$
	\eqref{co:Simonov-psd-pt0}
	\Leftrightarrow
	\eqref{co:Simonov-psd-pt1}
	$
	is clear. 
	The implication 
	$\eqref{co:Simonov-psd-pt00}
	\Leftarrow
	\eqref{co:Simonov-psd-pt0}
	$
	is trivial, 
	while the implication
	$\eqref{co:Simonov-psd-pt00}
	\Rightarrow
	\eqref{co:Simonov-psd-pt0}
	$
	follows by \cite[Theorem 5.1]{MS23+}.
	The equivalence 
	$\eqref{co:Simonov-psd-pt3}\Leftrightarrow
	\eqref{co:Simonov-psd-pt2}$
	follows by noticing that for $\ell=0,2$ we have
	$$
	\sum_{i,j=-n_1}^{n_2-1}v_{j+n_1}^TS_{i+j+\ell}v_{i+n_1}=
	\row(v_{i+n_1})_{i\in [-n_1;n_2-1]}(S_{i+j+\ell})_{i,j=-n_1}^{n_2-1}\col(v_{i+n_1})_{i\in [-n_1;n_2-1]}.
	$$
	Since $(S_{i+j})_{i,j=-n_1}^{n_2-1}$ is positive semidefinite,
	for $\ell=0,2$ we have
	$$
	\sum_{i,j=-n_1}^{n_2-1}v_{j+n_1}^TS_{i+j+\ell}v_{i+n_1}=0
	\quad\Leftrightarrow\quad
	\col(v_{i+n_1})_{i\in [-n_1;n_2-1]}\in \Ker (S_{i+j+\ell})_{i,j=-n_1}^{n_2-1}.
	$$
	
	It remains to prove the equivalence 
	$\eqref{co:Simonov-psd-pt1}\Leftrightarrow\eqref{co:Simonov-psd-pt3}$.
	Define a sequence $\widetilde{\cS} = (\widetilde{S}_0, \widetilde{S}_1, \ldots, \widetilde{S}_{2n_1+2n_2})$, where $\widetilde{S}_i := S_{i-2n_1}$. 
	
	First we prove the implication \eqref{co:Simonov-psd-pt1} $\Rightarrow$ \eqref{co:Simonov-psd-pt3}. Let $\mu := \sum_{j=1}^{\ell}\delta_{x_j}A_j$ be a minimal representing measure such that $S_i=\int_\RR x^i d\mu = \sum_{j=1}^{\ell}x_j^iA_j$ for each $i$. Clearly, $x_j \neq 0$ for each $j$
	and $(S_{i+j})_{i,j=-n_1}^{n_2}$ is positive semidefinite. Defining $\widetilde{\mu} := \sum_{j=1}^{\ell}\delta_{x_j}(x_j^{-2n_1}A_j)$, we have $$\widetilde{S}_i = S_{i-2n_1} = \sum_{j=1}^{\ell}x_j^{i-2n_1}A_j = \sum_{j=1}^{\ell}x_j^i(x_j^{-2n_1}A_j) = \int_\RR x^i d\widetilde{\mu}.$$
	By Theorem \ref{co:corollary-atom-avoidance}, used for $t = 0$,
	$\Ker (S_{i+j})_{i,j=-n_1}^{n_2-1}=\Ker (S_{i+j})_{i,j=-n_1+1}^{n_2}$.
	This proves  \eqref{co:Simonov-psd-pt1} $\Rightarrow$ \eqref{co:Simonov-psd-pt3}.
	
	It remains to prove the implication \eqref{co:Simonov-psd-pt3} $\Rightarrow$ \eqref{co:Simonov-psd-pt1}. By assumption,
	$M_{\widetilde{\cS}}$ is positive semidefinite.
	The assumption $\Ker (S_{i+j})_{i,j=-n_1}^{n_2-1}=\Ker (S_{i+j})_{i,j=-n_1+1}^{n_2}$ implies that
	$\cC(S_{i+j})_{i,j=-n_1}^{n_2-1}=\cC (S_{i+j})_{i,j=-n_1+1}^{n_2}$.
	In particular, 
	$\cC(\col(\widetilde{S}_{n_1+n_2+i})_{i\in [n_1+n_2]})\subseteq \cC(M_{\widetilde{\cS}}(n_1+n_2-1))$.
	By Theorem \ref{th:Hamburger-matricial}, $\widetilde\cS$ admits a representing measure. 
	By Theorem \ref{co:corollary-atom-avoidance}, $\widetilde{\cS}$ has a minimal representing measure $\widetilde{\mu} = \sum_{j=1}^{\ell}\delta_{x_j}A_j$ for some $x_j \in \RR \setminus \{0\}$ and $A_j \in \Sym_p^{\succeq 0}(\RR)$. Namely, $\widetilde{S}_i = \sum_{j=1}^{\ell}x_j^iA_j$ for each $i\in [0;2n_1+2n_2]$. But then
	$$S_i = \widetilde{S}_{i+2n_1} = \sum_{j=1}^{\ell}x_j^{i+2n_1}A_j = \sum_{j=1}^{\ell}x_j^ix_j^{2n_1}A_j,$$ whence $\mu := \sum_{j=1}^{\ell}\delta_{x_j}(x_j^{2n_1}A_j)$ is a minimal representing measure in \eqref{co:Simonov-psd-pt1}.

	\section{Example for Theorem \ref{th:mainTheorem}}
	\label{sec:example}
	
	The next example{\footnote{The \textit{Mathematica} file with numerical computations can be found on the link \url{https://github.com/ZobovicIgor/Matricial-Gaussian-Quadrature-Rules/tree/main}.}} shows how to apply Theorem \ref{th:mainTheorem} to construct a minimal representing measure $\mu$ for a linear operator $L:\RR[x]_{\leq 2n}\to \Sym_p(\RR)$ such that $\mult_{\mu}t >0$, where $t \in \RR$ is a prescribed number.
	
	Let $p=4$, $n=2$ and $\cS := (S_i)_{i\in [0;4]}$, where
	\begin{align*}
		S_0 &= 
		\begin{pmatrix}
			4 & 0 & 0 & 0 \\
			0 & 7 & 3 & 3 \\
			0 & 3 & 4 & 3 \\
			0 & 3 & 3 & 3
		\end{pmatrix}, \quad 
		&S_1 &= S_3 =
		\begin{pmatrix}
			0 & 0 & 0 & 0 \\
			0 & -1 & -1 & -1 \\
			0 & -1 & 0 & -1 \\
			0 & -1 & -1 & -1
		\end{pmatrix}, \quad
		&S_2 &= S_4 =
		\begin{pmatrix}
			2 & 0 & 0 & 0 \\
			0 & 3 & 1 & 1 \\
			0 & 1 & 2 & 1 \\
			0 & 1 & 1 & 1
		\end{pmatrix}.
	\end{align*}
	Let $L:\RR[x]_{\leq 4}\to \Sym_4(\RR)$ be a linear operator, defined by $L(x^i) = S_i$ for each $i\in [0;4]$.
	Since $M(2)$ is positive semidefinite and 
	$\cC(\col(S_3,S_4)) \subseteq \cC(M(1))$, by Theorem \ref{th:Hamburger-matricial}, $L$ has a representing measure. 
	It is easy to check that $\widetilde{\mu} := \sum_{j=1}^{3}\delta_{\widetilde{x}_j}\widetilde{A}_j$ is a minimal representing measure for $L$, where 
	\begin{align*}
		&\widetilde{x}_1 = -1, \quad \widetilde{x}_2 = 0, \quad \widetilde{x}_3 = 1, \quad
		\widetilde{A}_1 = 
		\begin{pmatrix}
			1
		\end{pmatrix} \oplus
		\begin{pmatrix}
			2 & 1 & 1 \\
			1 & 1 & 1 \\
			1 & 1 & 1
		\end{pmatrix}, \quad \widetilde{A}_2 = 2\widetilde{A}_1, \quad \widetilde{A}_3 = I_3 \oplus \begin{pmatrix} 0 \end{pmatrix}.
	\end{align*}
	Let $t = 1$ and $T_i$ be defined as in Theorem \ref{th:mainTheorem} for $i\in [0;4]$.
	A computation reveals that
	\begin{align*}
		T_0 &= 
		\begin{pmatrix}
			4
		\end{pmatrix} \oplus
		\begin{pmatrix}
			7 & 3 & 3 \\
			3 & 4 & 3 \\
			3 & 3 & 3
		\end{pmatrix}, \quad 
		&T_1 &= -4W,
			&T_2 &= 6W,
				&T_3 &= -10W,
					&T_4 &= 18W,
					\end{align*}
					where 
					$$W = \begin{pmatrix}
						1
					\end{pmatrix} \oplus
					\begin{pmatrix}
						2 & 1 & 1 \\
						1 & 1 & 1 \\
						1 & 1 & 1
					\end{pmatrix}.$$
					Define 
					$\cT:=(T_i)_{i\in [0;4]}$.
					Let $\Dep(X_{\cT}^i)$, $\Dep_1(X_{\cT}^i)$, $\Depn(X_{\cT}^i)$
					and $\Depn_{\ell}(X_{\cT}^i)$, $\ell=0,1$,
					and $A$, $B$
					be as in \eqref{eq:defDep-of-Y-i},  
					\eqref{eq:defDep-of-Y-i-alternative}, \eqref{eq:defDepnew-of-Y-i} and Theorem \ref{th:mainTheorem}, 
					respectively.
					We have
					\begin{align*}
						\Dep(X_{\cT}^0) &= \emptyset, &\Dep(X_{\cT}^1) &= \{4\}, &\Dep(X_{\cT}^2) &= \{3,4\}, \\
						\Dep_1(X_{\cT}^0) &= \emptyset, &\Dep_1(X_{\cT}^1) &= \{4\}, &\Dep_1(X_{\cT}^2) &= \{4\}, \\
						\Depn(X_{\cT}^0) &= \emptyset, &\Depn(X_{\cT}^1) &= \{4\}, &\Depn(X_{\cT}^2) &= \{3\}, \\
						\Depn_1(X_{\cT}^0) &= \emptyset, &
						\Depn_1(X_{\cT}^1)&= \{4\}, &
						\Depn_1(X_{\cT}^2)&= \emptyset, \\
						\Depn_0(X_{\cT}^0) &= \emptyset, &
						\Depn_0(X_{\cT}^1)&= \emptyset, &
						\Depn_0(X_{\cT}^2)&= \{3\},
					\end{align*}
					and $A = \{4\}$, $B= \{1,2\}$.
					Thus $\Card A = 1$ and $\Card(A\cup B)= 3$. It follows from Theorem \ref{th:mainTheorem} that for each $m$ with $1\leq m \leq 3$, there exists a representing measure $\mu$ for $L$ such that $\mult_\mu t = m$. We will construct such a measure for $m = 2$ (note that $\mult_{\widetilde{\mu}}t = \Rank \widetilde{A}_3 = 3$).
					
					Let $s, z_j, p_{j,\ell}, p_j, p_0$ be as in \eqref{eq:defmathcalZ-v2}, \eqref{def-of-p0}. By the computations above, we have 
					$s = 2$, 
					$z_1 = 1$, $z_2 = 2$, 
					$p_{1,0} = p_{2,1} = 1$, 
					$p_{1,1} = p_{2,0} = 0$, $p_1 = p_2 = 1$ 
					and 
					$p_0 = 2$. 
					Since the columns of each block column $X_{\cT}^i$ are already in the order \eqref{order-permutation-Q}, we choose the permutation matrix $Q = I_4$, whence $\widecheck{X}_{\cT}^i = X_{\cT}^iQ = X_{\cT}^i$ for each $i$. Decompose $\widecheck{X}_{\cT}^i$ as in \eqref{block-decomposition-of-widehatYzj-semidefinite-case} for each $i$. Next we will compute the matrices \eqref{relation-matrices} satisfying \eqref{column-relation-semidefinite-case} by the formulas:
					\begin{align*}
						\col(
						\widecheck{B}_{0}^{(1,1)},
						\widecheck{B}_{1,(0)}^{(1,1)},
						\widecheck{B}_{1}^{(1,1)})
						&= \begin{pmatrix}
							\widecheck{X}_{\cT}^0 & 
							(\widecheck{X}_{\cT}^1)_{(0)} & (\widecheck{X}^1_{\cT})_{(1,0)} & (\widecheck{X}^1_{\cT})_{(1,1)} &(\widecheck{X}^1_{\cT})_{(2,0)}
						\end{pmatrix}^\dag (\widecheck{X}_{\cT}^1)_{(2,1)} \\
						&= \col(\col(0,0,0,0),\col(0,0),1), \\
						\col(
						\widecheck{B}_{0}^{(2,0)},
						\widecheck{B}_{1}^{(2,0)},
						\widecheck{B}_{2,(0)}^{(2,0)})
						&= \begin{pmatrix}
							\widecheck{X}_{\cT}^0 & \widecheck{X}_{\cT}^1 & (\widecheck{X}^2_{\cT})_{(0)}
						\end{pmatrix}^\dag (\widecheck{X}_{\cT}^2)_{(1,0)} \\
						&= \col(\col(0,0,0,-2),\col(0,0,-\frac{3}{2},-\frac{3}{2}),\col(0,0)).
					\end{align*}
					Hence, the matrices 
					$\widecheck{H}_{0}^{(j)}$,
					$\widecheck{H}_{1}^{(j)}$,
					$\widecheck{H}_{2,(0)}^{(j)}$,
					$\widecheck{H}_{2}^{(j)}$
					for
					$j=1,2$,
					in
					\eqref{column-relation-v3-semidefinite-case}
					are equal to:
					\begin{align*}
						&\widecheck{H}_{0}^{(1)} = \mathbf{0}_{4\times 2}, \quad
						\widecheck{H}_{1}^{(1)} = \begin{pmatrix}
							\mathbf{0}_{4\times 1} & \widecheck{B}_{0}^{(1,1)}
						\end{pmatrix} = \mathbf{0}_{4\times 2}, \\
						&\widecheck{H}_{2,(0)}^{(1)} = \begin{pmatrix}
							\mathbf{0}_{2\times 1} & \widecheck{B}_{1,(0)}^{(1,1)}
						\end{pmatrix} = \mathbf{0}_{2}, \quad \widecheck{H}_{2}^{(1)} = \begin{pmatrix}
							0 & -\widecheck{B}_{1}^{(1,1)} \\[0.2em]
							0 & 1
						\end{pmatrix} = \begin{pmatrix}
							0 & -1 \\
							0 & 1
						\end{pmatrix}, \\
						&\widecheck{H}_{0}^{(2)} = \begin{pmatrix}
							\widecheck{B}_{0}^{(2,0)} & \mathbf{0}_{4\times 1}
						\end{pmatrix} = \row(\col(0,0,0,-2),\mathbf{0}_{4\times 1}), \\ &\widecheck{H}_{1}^{(2)} = \begin{pmatrix}
							\widecheck{B}_{1}^{(2,0)} & \mathbf{0}_{4\times 1}
						\end{pmatrix} = -\frac{3}{2}\row(\col(0,0,1,1),\mathbf{0}_{4\times 1}), \\
						&\widecheck{H}_{2,(0)}^{(2)} = \begin{pmatrix}
							\widecheck{B}_{2,(0)}^{(2,0)} & \mathbf{0}_{2\times 1}
						\end{pmatrix} = \mathbf{0}_{2}, \quad \widecheck{H}_{2}^{(2)} = \begin{pmatrix}
							1 & 0 \\ 0 & 0
						\end{pmatrix}.
					\end{align*}
					We now compute
					\begin{align*}
						\widecheck{H}_{0} &= \widecheck{H}_{0}^{(1)} + \widecheck{H}_{0}^{(2)} = 
						\row(\col(0,0,0,-2),\mathbf{0}_{4\times 1}),\\
						\widecheck{H}_{1} &= \widecheck{H}_{1}^{(1)} + \widecheck{H}_{1}^{(2)} = 
						-\frac{3}{2}\row(\col(0,0,1,1),\mathbf{0}_{4\times 1}),\\
						\widecheck{H}_{2,(0)} &= \begin{pmatrix}
							\widecheck{B}_{2,(0)}^{(2,0)} & \widecheck{B}_{1,(0)}^{(1,1)}
						\end{pmatrix} = \mathbf{0}_2, \qquad
						\widecheck{H}_{2}= \begin{pmatrix}
							1 & -\widecheck{B}_1^{(1,1)} \\[0.2em]
							0 & 1
						\end{pmatrix} = \begin{pmatrix}
							1 & -1 \\
							0 & 1
						\end{pmatrix}, \\
						H_0 &= Q\widecheck{H}_{0} = \widecheck{H}_{0} = 
						\row(\col(0,0,0,-2),\mathbf{0}_{4\times 1}),\\
						H_1 &= Q\widecheck{H}_{1} = 
						-\frac{3}{2}\row(\col(0,0,1,1),\mathbf{0}_{4\times 1}),\\
						G &= 
						\col(
						H_0,
						H_1,
						\widecheck{H}_{2,(0)}).
					\end{align*}
					Define $\mathcal{H}_0$, $\mathcal{H}$ and $K$ as in \eqref{def-mathcalH0-etc}:
					\begin{align*}
						\mathcal{H}_0 = \begin{pmatrix}
							T_0 \\
							T_1
						\end{pmatrix}, \quad \mathcal{H} = \begin{pmatrix}
							T_1 & T_2 \\
							T_2 & T_3
						\end{pmatrix}, \quad K = \begin{pmatrix}
							T_3 \\
							T_4
						\end{pmatrix}.
					\end{align*}
					Note that $$k = \Rank \begin{pmatrix}
						\mathcal{H} & K
					\end{pmatrix} - \Rank \mathcal{H} = 6 - 6 = 0.$$
					Therefore we can choose $P_1 := I_2$ and hence $P = Q(P_1 \oplus I_2) = I_4$. We decompose $\widehat{K} = KP = K$ as in \eqref{def-widehatK} (using $p=4$, $p_0=2$, $m=2$, $\Card A=1$):
					\begin{align*}
						\widehat{K}_1 &=
						\col(-10,0,0,0,18,0,0,0),\\
						\widehat{K}_2 &=
						\col(0,-20,-10,-10,0,36,18,18),\\
						\widehat{K}_3 &=
						\row(
						\col(0,-10,-10,-10,0,18,18,18),
						\col(0,-10,-10,-10,0,18,18,18)).
					\end{align*}
					Computing $J$ satisfying \eqref{definition-hat-K-2-psd} by
					$$J := \begin{pmatrix}
						\mathcal{H} & \widehat{K}_1
					\end{pmatrix}^\dag \widehat{K}_2 = 
					\col(0,-2,0,0,0,-3,0,0,0),$$
					and computing 
					\begin{align*}
						\widehat{G} &= (I_8 \oplus P_1^T)G\widecheck{H}_{2}^{-1} \\ &= \row\Big(\col(0,0,0,-2,0,0,-\frac{3}{2},-\frac{3}{2},0,0),\col(0,0,0,-2,0,0,-\frac{3}{2},-\frac{3}{2},0,0)\Big),
					\end{align*}
					we have 
					\begin{align*}
						\widehat{K}_2 = \begin{pmatrix}
							\mathcal{H} & \widehat{K}_1
						\end{pmatrix}J \quad \text{and} \quad \widehat{K}_3 = \begin{pmatrix}
							\mathcal{H} & \widehat{K}_1 & \widehat{K}_2
						\end{pmatrix}\widehat{G}.
					\end{align*}
					Let $\widehat{Z}_1 := \begin{pmatrix}
						0
					\end{pmatrix}$, use \eqref{equality-v3-semidefinite-case} to compute 
					$$\widehat{Z}_2 = \begin{pmatrix}
						0
					\end{pmatrix},\; 
					\widehat{Z}_3 = \begin{pmatrix}
						0 & 0
					\end{pmatrix},\; 
					\widehat{Z}_4 = \begin{pmatrix}
						-68
					\end{pmatrix},\; 
					\widehat{Z}_5 = \begin{pmatrix}
						-34 & -34
					\end{pmatrix},\; 
					\widehat{Z}_6 = \begin{pmatrix}
						-34 & -34 \\
						-34 & -34
					\end{pmatrix}$$ and define $$\widehat{Z} \equiv \begin{pmatrix}
						\widehat{Z}_1 & \widehat{Z}_2 & \widehat{Z}_3 \\[0.2em]
						\widehat{Z}_2^T & \widehat{Z}_4 & \widehat{Z}_5 \\[0.2em]
						\widehat{Z}_3^T & \widehat{Z}_5^T & \widehat{Z}_6 
					\end{pmatrix} = \begin{pmatrix}
						0 & 0 & 0 & 0 \\
						0 & -68 & -34 & -34 \\
						0 & -34 & -34 & -34 \\
						0 & -34 & -34 & -34
					\end{pmatrix}.$$ We can easily check that the matrix $\widehat{Z}_1$ satisfies \eqref{choice-of-hat-Z1-semidefinite-case}. Let $\widehat{T}_2 = T_2P = T_2$. Next we compute matrices $U_1$ and $U_2$ by \eqref{061025-0840}:
					\begin{align*}
						\begin{pmatrix}
							U_1 \\
							U_2
						\end{pmatrix} = \begin{pmatrix}
							\mathcal{H}_0 & \mathcal{H} \\[0.2em]
							\widehat{T}_2 & \widehat{K}^T
						\end{pmatrix}^\dag \begin{pmatrix}
							\widehat{K}_1 \\
							\widehat{Z}_1 \\
							\widehat{Z}_2^T \\
							\widehat{Z}_3^T
						\end{pmatrix} = 
						\col(17,0,0,0,66,0,0,0,31,0,0,0).
					\end{align*}
					Write $J = \begin{pmatrix}
						J_1 \\
						J_2
					\end{pmatrix}$, where 
					$J_1 = 
					\col(0,-2,0,0,0,-3,0,0)$
					and $J_2 = \begin{pmatrix}
						0
					\end{pmatrix}$, and further decompose
					\begin{align*}
						J_1 = \begin{pmatrix}
							J_{1,0} \\
							J_{1,1}
						\end{pmatrix}, \quad U_2 = \begin{pmatrix}
							U_{2,0} \\
							U_{2,1}
						\end{pmatrix}, \quad \widehat{G} = \begin{pmatrix}
							\widehat{G}_1 \\
							\widehat{G}_2 \\
							\widehat{G}_3
						\end{pmatrix},
					\end{align*}
					as in \eqref{061025-0843}, where
					\begin{align*}
						J_{1,0} &= 
						\col(0,-2,0,0),\quad
						J_{1,1} = 
						\col(0,-3,0,0)
						,\\ 
						\quad U_{2,0} &= 
						\col(66,0,0,0),
						\quad U_{2,1} = 
						\col(31,0,0,0),\\
						\widehat{G}_1 &= 
						\row\big(
						\col(0,0,0,-2,0,0,-\frac{3}{2},-\frac{3}{2}),
						\col(0,0,0,-2,0,0,-\frac{3}{2},-\frac{3}{2})\big),\\
						\widehat{G}_2 &= \widehat{G}_3=\begin{pmatrix}
							0 & 0
						\end{pmatrix}. 
					\end{align*}
					As in \eqref{def:Q-n}, \eqref{def:Qi}, \eqref{061025-0853} define
					\begin{align*}
						&Q_0 := \begin{pmatrix}
							U_{2,0} & J_{1,0} & H_0\widecheck{H}_n^{-1}
						\end{pmatrix} =
						\begin{pmatrix}
							66 & 0 \\
							0 & -2
						\end{pmatrix} \oplus
						\begin{pmatrix}
							0 & 0 \\
							-2 & -2
						\end{pmatrix},\\
						&Q_1 := \begin{pmatrix}
							U_{2,1} & J_{1,1} & H_1\widecheck{H}_n^{-1}
						\end{pmatrix} = 
						\begin{pmatrix}
							31 & 0 \\
							0 & -3
						\end{pmatrix} \oplus
						\begin{pmatrix}
							-\frac{3}{2} & -\frac{3}{2} \\[0.3em]
							-\frac{3}{2} & -\frac{3}{2}
						\end{pmatrix},\\
						&Q_2 := P\begin{pmatrix}
							1 & -J_2 & -\widehat{G}_2 \\
							0 & 1 & -\widehat{G}_3 \\
							\mathbf{0}_{2,1} & \mathbf{0}_{2,1} & I_2
						\end{pmatrix} = I_4.
					\end{align*}
					Let $Z := P\widehat{Z}P^T = Z$, $T_5 := Z$ and $T_6 := \row(T_i)_{i\in [3;5]}(M_\cT(2))^\dagger
					\col(T_i)_{i\in [3;5]}$.
					The matrix polynomial $H(x)$ (see \eqref{gen-poly-semidefinite-case}), defined by
					\begin{align*}
						H(x) &= 
						(x-1)^3Q_2 - (x-1)^2Q_1 - (x-1)Q_0 - 
						\begin{pmatrix}
							U_1 & \mathbf{0}_{4,3} 
						\end{pmatrix}
					\end{align*}
					is a block column relation of 
					$M_{\cS^{(e)}}(3)$,
					where $\cS^{(e)}=(S_i)_{i\in [0;6]}$ is as in 
					\eqref{061025-0900}.
					Its determinant is the following:
					\begin{align*}
						\det H(x) = x^2(x-1)^5(x+1)^2(x^3 - 34x^2 - x + 17).
					\end{align*}
					The atoms of the minimal representing measure $\mu$ for $L$ are the zeroes of $\det H(x)$. Namely, writing $\mu = \sum_{j=1}^{6}\delta_{x_j}A_j$, we have \begin{align*}
						x_1 = -1, \quad x_2 = 0, \quad x_3 = 1, \quad x_4 \approx -0.7143, \quad x_5 \approx -0.6996, \quad x_6 \approx 34.0147.
					\end{align*}
					It remains to find the masses $A_j$ by computing
					\begin{align*}
						\col(A_i)_{i\in [6]}=
						\left (V_{(x_i)_{i\in [6]}}^{-1} \otimes I_4\right )
						\col(S_i)_{i\in [0;5]},
					\end{align*}
					where $V_{(x_i)_{i\in [6]}} = \left (x_j^{i-1}\right )_{i,j=1}^{6}$ is the Vandermonde matrix with $i$-th row equal to $\row(x_j^{i-1})_{j\in [6]}$.
					The masses are:
					\begin{align*}
						&A_1 = 
						\mathbf{0}_1 \oplus
						\begin{pmatrix}
							2 & 1 & 1 \\
							1 & 1 & 1 \\
							1 & 1 & 1
						\end{pmatrix}
						\quad A_2 = 
						2A_1,
						\quad A_3 = 
						\begin{pmatrix}
							0 & 0 \\
							0 & 1
						\end{pmatrix} \oplus
						\begin{pmatrix}
							1 & 0 \\
							0 & 0
						\end{pmatrix},\\          
						&A_4 \approx 
						\begin{pmatrix}
							1.9792
						\end{pmatrix} \oplus \mathbf{0}_3,
						\quad A_5 \approx
						\begin{pmatrix}
							2.0208
						\end{pmatrix} \oplus \mathbf{0}_3,
						\\
						&A_6 \approx
						\begin{pmatrix}
							7.4734 \cdot 10^{-7}
						\end{pmatrix} \oplus \mathbf{0}_3,
					\end{align*}
					We see that $x_3 = t = 1$ and $\Rank A_3 = 2$, 
					whence $\mult_\mu t = m = 2$. 
					
					\begin{remark}
						\label{re:remarkForExample}
						\begin{enumerate}[leftmargin=*]
							\item
							\label{re:remarkForExample}
							The corresponding masses $A_j$ for the constructed measure $\mu$ are obtained via
							$$\begin{pmatrix}
								\col(A_i)_{i\in [\ell]}
							\end{pmatrix}^T = \left (V_{(x_i)_{i\in [\ell]}}^{-1} \otimes I_p\right )
							\col(S_i)_{i\in [0;\ell-1]}$$ as in \cite[Remark 3.2(3)]{ZZ25}. If $\ell > 2n+3$, then not all $S_i$ are given. Computing 
							\medskip
							$$T_r = \row(T_i)_{i\in [r-n-1;r-1]}\col(\begin{pmatrix}
								U_1 & \mathbf{0}_{p\times (m-\Card A+p-p_0)}\end{pmatrix}Q_n^{-1}), \col(Q_{i-1}Q_n^{-1})_{i\in[n]})$$
							\medskip
							for $ r=2n+3,2n+4,\ldots,\ell - 1$, where $U_1$ is as in \eqref{061025-0840} and $Q_i$ are as in \eqref{def:Qi} and \eqref{061025-0853},
							and defining $$\cT^{(e_1)}:=(T_i)_{i\in [0;\ell-1]},$$ we can compute the missing moments $S_{2n+3}, S_{2n+4}, \ldots, S_{\ell - 1}$ by $$S_{r}:=L_{\cT^{(e_1)}}((x+t)^{r}),\quad r=2n+3,2n+4,\ldots,\ell - 1.$$
							\item
							The matrix polynomial $W(x) := x^{n+1}I_p - \sum_{i=0}^{n}x^iW_i$, where
							$$\col(W_i)_{i \in [0;n]} := M_{\cS}(n)^\dag \col(S_i)_{i \in [n+1;2n+1]}$$
							is also a block column relation of $M_{\cS^{(e)}}(n+1)$ and $\cS^{(e)}$ is as in \eqref{061025-0900}. However, it is not clear from the factorization of the determinant of $W(x)$, which zeroes of $\det W(x)$ are the atoms of the measure $\mu$. For example, using the data given in the example above, we have
							$$\col(W_i)_{i \in [0;2]} := M_{\cS}(2)^\dag \col(S_i)_{i \in [3;5]},$$ which gives us $$W_0 = \begin{pmatrix}
								-17 & 0 \\
								0 & 0
							\end{pmatrix} \oplus \begin{pmatrix}
								\frac{1}{5} & 0 \\
								-\frac{1}{5} & 0
							\end{pmatrix}, \; W_1 = I_2 \oplus \begin{pmatrix}
								\frac{13}{20} & \frac{1}{4} \\
								\frac{1}{4} & \frac{1}{4}
							\end{pmatrix}, \; W_2 = \begin{pmatrix}
								34 & 0 \\
								0 & 0
							\end{pmatrix} \oplus \begin{pmatrix}
								\frac{3}{20} & -\frac{1}{4} \\
								-\frac{1}{4} & -\frac{1}{4}
							\end{pmatrix}.$$
							Defining the matrix polynomial $$W(x) = x^3I_4 - x^2W_2 - xW_1 - W_0,$$ we have $$\det W(x) = \frac{1}{10}x^2(x-1)^2(x+1)^2(x^3-34x^2-x+17)(10x^3+x^2-1).$$
							We can see that all of the atoms of the measure $\mu$ constructed in the example are the zeroes of $\det W(x)$, but just using the matrix polynomial $W(x)$ to find the atoms of the representing measure for the given linear operator, we have no direct way of knowing that the zeroes of the factor $10x^3+x^2-1$ are not the atoms of the measure.
						\end{enumerate}
					\end{remark}
	
	


\begin{thebibliography}{scw69}
		
		
		
		
		
		
		\bibitem[Alb69]
		{Alb69}
		A.\ Albert, 
		\textit{Conditions for positive and nonnegative definiteness in terms of pseudoinverses},
		SIAM J. Appl. Math. \textbf{17} (1969), 434--440.
		
		
		
		
		\bibitem[BKRSV20]
		{BKRSV20}
		G.\ Blekherman, M.\ Kummer, C.\ Riener, M.\ Schweighofer, C.\ Vinzant,
		\textit{Generalized eigenvalue methods for Guassian quadrature rules}.
		{Ann.\ H.\ Lebesgue} \textbf{3} (2020), 1327--1341.
		
		\bibitem[BW11]
		{BW11}
		M.\ Bakonyi, H.J.\ Woerdeman, 
		\textit{Matrix Completions, Moments, and Sums of Hermitian Squares}, 
		Princeton University Press, Princeton, 2011.
		
		\bibitem[Bol96]
		{Bol96}
		V.\ Bolotnikov, 
		\textit{On degenerate Hamburger moment problem and extensions
			of nonnegative Hankel block matrices}, 
		Integral Equ.\ Oper.\ Theory
		\textbf{25} (1996), 253--276.
		
		
		
		
		
		\bibitem[CF91]
		{CF91}
		R.\ Curto, L.\ Fialkow,
		\textit{Recursiveness, positivity, and truncated moment problems.}
		{Houston J.\ Math.}
		\textbf{17} (1991), 603--635.
		
		
		
		
		
		
		
		\bibitem[CF05]
		{CF05}
		R.\ Curto, L.\ Fialkow,
		\textit{Solution of the truncated hyperbolic moment problem},
		Integral Equ.\ Oper.\ Theory \textbf{52} (2005), 181--218.
		
		
		
		
		
		
		
		
		
		
		
		
		
		
		
		\bibitem[DS03]
		{DS03}
		H.\ Dette, W.J.\ Studden,
		\textit{Quadrature formulas for matrix measures—a geometric approach},
		Linear Algebra Appl.\
		\textbf{364}
		(2003),
		33--64.
		
		\bibitem[DD02]{DD02}
		A.J.\ Dur\'an, E.\ Defez,
		\textit{Orthogonal matrix polynomials and quadrature formulas},
		Linear Algebra Appl.\
		\textbf{345} 
		(2002),
		71--84.
		
		\bibitem[DLR96]{DLR96}
		A.J.\ Dur\'an, P. Lopez-Rodriguez,
		\textit{Orthogonal Matrix Polynomials: Zeros and Blumenthal's Theorem},
		J.\ Approx.\ Theory
		\textbf{84} 
		(1996),
		96--118.
		
		\bibitem[Dym89]{Dym89}
		H.\ Dym, 
		\textit{On Hermitian block Hankel matrices, matrix polynomials, the
			Hamburger moment problem, interpolation and maximum entropy}, 
		Integral Equ.\ Oper.\ Theory \textbf{12} (1989) 757--812.
		
		\bibitem[DFKM09]
		{DFKMT09}
		Y.M.\ Dyukarev, B.\ Fritzsche, B.\ Kirstein, B., C.\ M\"adler,
		H.C.\ Thiele, \textit{On distinguished solutions of truncated matricial Hamburger
			moment problems}, 
		Complex Anal.\ Oper.\ Theory \textbf{3} (2009) 759--834.
		
		
		\bibitem[FKM24]{FKM24}
		B.\ Fritzsche, B.\ Kirstein, C.\ M\"adler,
		\textit{Description of the set of all possible masses at a fixed point of solutions of a truncated matricial Hamburger moment problem},
		Linear Algebra Appl.\ \textbf{697} (2024),
		639--687.
		
		
		
		
		
		
		\bibitem[GLR82]
		{GLR82}
		I.\ Gohberg, P.\ Lancaster, L.\ Rodman, 
		\textit{Matrix polynomials}, 
		Computer Science and Applied Mathematics. 
		Academic Press, Inc., New York-London,
		1982.
		
		
		
		
		
		
		
		
		
		\bibitem[KT22]
		{KT22}
		D.P.\ Kimsey, M.\ Trachana:
		\textit{On a solution of the multidimensional truncated matrix-valued moment problem},
		{Milan J.\ Math.} \textbf{90} (2022), 17--101.
		
		
		
		
		
		
		
		
		\bibitem[MS24a]
		{MS23+}
		C.\ M\"adler, K.\ Schm\"udgen, 
		\textit{On the Truncated Matricial Moment Problem. I},
		J.\ Math.\ Anal.\ Appl.\ \textbf{540} (2024), 36 pp.
		
		\bibitem[MS24b]
		{MS23++}
		C.\ M\"adler, K.\ Schm\"udgen,
		\textit{On the Truncated Matricial Moment Problem. II}, 
		{Linear Algebra Appl.} \textbf{649} (2024), 63--97.
		
		
		
		
		
		
		
		
		
		\bibitem[NZ25]
		{NZ25}
		R.\ Nailwal, A.\ Zalar,
		\textit{The truncated univariate rational moment problem}, 
		Linear Algebra Appl.\ \textbf{708}
		(2025), 280--301.
		
		\bibitem[NZ+]
		{NZ+} 
		R.\ Nailwal, A.\ Zalar, 
		\textit{Moment theory approach to Gaussian quadratures with prescribed nodes},
		arXiv preprint \url{https://arxiv.org/abs/2412.20849}.
		
		
		
		
		
		
		
		\bibitem[Sch87]
		{Sch87}
		K.\ Schm\"udgen,
		\textit{On a generalization of the classical moment problem},
		J.\ Math.\ Anal.\ Appl.\ \textbf{125} (1987), 461--470.
		
		%
		
		
		\bibitem[Sim06]
		{Sim06}
		K.K.\ Simonov: 
		\textit{Strong truncated matrix moment problem of Hamburger}, 
		{Sarajevo J.\ Math.} 2(15) (2006), 181--204.
		
		
		
		
		\bibitem[Wol]
		{Wol}
		Wolfram Research, Inc., Mathematica, Version 14.0, Wolfram Research, Inc., Champaign,
		IL, 2025.
		
		
		
		\bibitem[Zal22a]
		{Zal22a}
		A.\ Zalar,
		\textit{The truncated moment problem on the union of parallel lines}.
		{Linear Algebra Appl.} \textbf{649} (2022), 186--239.
		
		\bibitem[Zal22b]
		{Zal22b}
		A.\ Zalar,
		\textit{The strong truncated Hamburger moment problem with and without gaps}.
		{J.\ Math.\ Anal.\ Appl.} \textbf{516} (2022), 21pp.
		
		
		
		\bibitem[ZZ25]
		{ZZ25}
		A.\ Zalar, I.\ Zobovi\v{c},
		\textit{Matricial Gaussian quadrature rules: nonsingular case},
		Linear Algebra Appl. \textbf{731} (2026), 160--185.
		
		\bibitem[Zha05]
		{Zha05}
		F. Zhang, \textit{The Schur Complement and Its Applications}. 
		Springer-Verlag, New York, 2005.
		
		
	\end{thebibliography}
\end{document}